\documentclass[12pt,a4paper,pdflatex]{amsart}

\usepackage[latin1]{inputenc}
\usepackage[colorlinks]{hyperref}

\usepackage{amssymb}
\usepackage{amsmath}
\usepackage{amsthm}
\usepackage{color}
\usepackage{url}
\usepackage{tikz}
\usepackage{enumitem}
\usepackage{caption}
\usepackage{subcaption}
\usepackage[kerning,spacing]{microtype}
\microtypecontext{spacing=nonfrench}
\theoremstyle{theorem}
\newtheorem{theorem}{Theorem}
\newtheorem{corollary}[theorem]{Corollary}
\newtheorem{lemma}[theorem]{Lemma}
\newtheorem{proposition}[theorem]{Proposition}

\theoremstyle{definition}

\newtheorem{example}[theorem]{Example}

\newtheorem{remark}[theorem]{Remark}

\DeclareMathOperator{\cone}{cone}
\DeclareMathOperator{\conv}{conv}
\DeclareMathOperator{\aff}{aff}
\DeclareMathOperator{\Hom}{Hom}
\DeclareMathOperator{\mult}{mult}

\DeclareMathOperator{\Q}{{\mathbb{Q}}}
\DeclareMathOperator{\Z}{{\mathbb{Z}}}
\DeclareMathOperator{\R}{{\mathbb{R}}}

\DeclareMathOperator{\F}{{\mathbb{F}}} 

\DeclareMathOperator{\nib}{nib}
\DeclareMathOperator{\hilb}{Hilb}
\DeclareMathOperator{\height}{ht}
\renewcommand{\L}{\mathcal{L}}
\renewcommand{\P}{\mathbb{P}}
\newcommand{\nbr}{n}
\newcommand{\OO}{\mathcal{O}}

\providecommand\polymake{\texttt{polymake}\ }
\providecommand\normaliz{\texttt{normaliz2}\ }
\providecommand\lattE{\texttt{Latte}\ }
\providecommand\fourtitwo{\texttt{4ti2}\ }

\title[Finitely many smooth $d$-polytopes with $\nbr$ lattice points]{Finitely many smooth
  $\boldsymbol d$-polytopes with $\boldsymbol \nbr$~lattice points}

\author[9 authors]{Tristram Bogart}
\address{Tristram Bogart, Universidad de los Andes, Colombia}
\author[]{Christian Haase}
\address{Christian Haase, Goethe-Universit\"at
  Frankfurt, Germany}
\author[]{Milena Hering}
\address{Milena Hering, University of Edinburgh, UK}
\author[]{Benjamin Lorenz}
\address{Benjamin Lorenz, Goethe-Universit\"at Frankfurt, Germany}
\author[]{Benjamin Nill}
\address{Benjamin Nill, Stockholms Universitet, Stockholm, Sweden}
\author[]{Andreas Paffenholz}
\address{Andreas Paffenholz, Technische Universit\"at Darmstadt, Germany}
\author[]{G\"unter Rote}
\address{G\"unter Rote, Freie Universit\"at Berlin, Germany}
\author[]{Francisco Santos}
\address{Francisco Santos, Universidad de Cantabria, Santander, Spain}
\author[]{Hal Schenck}
\address{Hal Schenck, University of Illinois, Urbana, IL, USA}

\date{July 17, 2013}

\begin{document}

\begin{abstract}
  We prove that for fixed $\nbr$ there are only finitely many 
  embeddings of $\Q$-factorial toric varieties $X$ into $\P^\nbr$ that
  are induced by a complete linear system.  
  The proof is based on a combinatorial result that 
  implies that for fixed nonnegative integers $d$ and $\nbr$, there
  are only finitely many smooth $d$-polytopes with $\nbr$ lattice
  points. We also enumerate all smooth $3$-polytopes with $\le 12$
  lattice points. 
\end{abstract}
 
\maketitle

\section{Introduction}
\label{sec:intro}

The present paper has two target audiences: combinatorialists and
algebraic geometers. We give combinatorial proofs of results motivated
by the algebraic geometry of toric varieties.
We provide two introductions with statements of the main results in
the language of divisors on toric varieties on the one hand, and in
the language of lattice polytopes on the other.
In Section~\ref{preliminaries}, we collect the relevant entries from
the dictionary translating between the two worlds. In
Section~\ref{sec:theorems} we prove our theorems, and in
Section~\ref{sec:classification} we report on first classification
results.

\subsection{Introduction (for algebraic geometers)}
\label{sec:intro:alg}

The purpose of this paper is to show the following finiteness theorem
about embeddings of toric varieties\footnote{All toric
varieties appearing in this paper are normal by construction.}
into projective space of a fixed dimension $\nbr$.

\begin{theorem}\label{thm:Qfactorialembeddings}
Let $\nbr$ be a nonnegative integer. Then there exist only 
finitely many embeddings of $\Q$-factorial
toric varieties into $\P^{\nbr}$ that are induced by a 
complete linear series. 
\end{theorem}

Neither of the two conditions ($\Q$-factorial and embedded via a complete linear series) 
can be omitted in the statement:

\begin{itemize}
\item For complete linear series embeddings of 
non-$\Q$-factorial varieties of dimension at least three, the embedding
dimension does not even bound the degree, see Example~\ref{egBruns}.
\item Every Hirzebruch surface $X(\F_a)$ 
(which is smooth, hence $\Q$-factorial)
admits an embedding into $\P^{5}$; see Example~\ref{Hirzebruch}.
\end{itemize}

Furthermore, there exist infinitely many polarized toric varieties $(X,L)$ where 
$X$ is $\Q$-factorial and $L$ is ample with $h^0(X,L) = 3$, see Example~\ref{singex}.

When we assume that $X$ is smooth we obtain an even stronger result. 
For an ample line bundle $L= \mathcal{O}(D)$ on a toric variety $X$, we let $d(L) = \sum D\cdot C$, where the sum runs over all torus invariant curves $C$ in $X$.   

\begin{theorem}\label{thm:smooth}
For fixed $\nbr$, there are only finitely many smooth polarized toric 
varieties $(X,L)$ such that $d(L) \leq \nbr$. 
\end{theorem}

Example \ref{ex:multnec} shows that this theorem is false for very ample line bundles on  $\Q$-factorial toric surfaces. 

In Section~\ref{sec:classification} we
use Oda's classification of smooth 3-dimensional 
toric varieties that are minimal with respect to equivariant blow-ups
to classify all
embeddings of smooth 3-dimensional toric varieties into $\P^{\le 11}$
using a complete linear series.  In the appendix we present the
complete list of the corresponding 3-polytopes with $\le 12$ lattice
points up to equivalence. 

Our motivation for this classification is a hierarchy of
long standing open questions on toric embeddings, for example Oda's question  \cite{Oda97} whether an ample line bundle on a smooth projective toric variety is normally generated
(see Section~\ref{sec:conjectures}).

\subsection{Introduction (for polyhedral geometers)}
\label{sec:intro:comb}
The purpose of this paper is to show that there is only a finite
number of classes (modulo integral equivalence) 
of smooth lattice
polytopes once we fix some properties of them.
For example, let us call a lattice polytope \emph{smooth} if it is
simple and all its normal cones (equivalently, all its tangent cones)
are unimodular. Then Theorem~\ref{thm:smooth} is equivalent to:

\begin{theorem}
\label{thm:polytopes}
  Let $\nbr$ be a nonnegative integer. Then, modulo integral equivalence, there are only finitely many smooth lattice
  polytopes with $\nbr$ lattice points on their edges.
\end{theorem}

We prove several versions of this theorem; the most general one
(Theorem~\ref{thm:points2}) says that instead of requiring our
polytopes to be smooth, as in the above in
Theorem~\ref{thm:polytopes}, it suffices to fix a finite list of
possible tangent cones for the vertices (modulo integral 
equivalence).

Our proofs are based on a statement that transfers finiteness from
dimension two to dimension $n$ (Lemma~\ref{lem:2dtoalld}), together
with a detailed analysis of the case of dimension two. In dimension two,
simply using Pick's theorem already implies that there is a finite
number of polygons with a fixed number of lattice points (see the proof
of Theorem~\ref{thm:points}), but by using the classification of
2-dimensional unimodular fans we get that it is in fact enough to fix
the number of lattice points on edges, as long as the \emph{multiplicity} of
the tangent cones is also bounded (Theorem~\ref{thm:mainforpolygons},
see Section~\ref{def-multiplicity} for the definition of {multiplicity}). In
the smooth case, we also give bounds on how many polygons there are
and how big their area can be in terms of the number of lattice points
on edges (Theorems~\ref{thm:fewpolygons} and~\ref{thm:smallpolygons}).

In Section~\ref{sec:classification} we use Oda's classification of
3-dimensional unimodular fans with $\le 8$ rays that are minimal with
respect to stellar subdivisions to classify all 3-dimensional
polytopes with unimodular normal fan and $\le 12$ lattice points, up
to equivalence. They are listed in the appendix. In subsequent work,
Anders Lundman has extended this classification to $16$ lattice
points~\cite{LundmanMSc}.

Also from the combinatorial viewpoint, our motivation for this
classification is a hierarchy of long standing open questions about
smooth polytopes (see Section~\ref{sec:conjectures}). 

\subsection{Related Results}

Let us briefly give an overview of related finiteness and
classification results.

The first finiteness theorem goes back to Hensley~\cite{Hensley}, with
the current best bound due to Pikhurko~\cite[{(9)}]{Pikhurko}.

\begin{theorem} 
  \label{thm:Hensley}
  For a positive integer $d$, there is a bound $V(d)$ so that the
  volume of every lattice $d$-polytope with $k \ge 1$ interior lattice
  points is bounded by $k \cdot V(d)$.
\end{theorem}

The second result, due to Lagarias and
Ziegler~\cite[Theorem~2]{lagariasZiegler}, implies that bounding the
volume automatically bounds the number of lattice points.

\begin{theorem}
  \label{thm:LZ}
  A family of lattice $d$-polytopes with bounded volume contains only
  a finite number of integral equivalence classes.
\end{theorem}

Putting these two results together we get:

\begin{corollary} \label{cor:LZ}
  Any family of lattice polytopes with bounded number of lattice
  points contains only finitely many integral equivalence classes of
  polytopes with interior lattice points.
\end{corollary}

\begin{example}\label{singex}
Without the assumption on interior lattice points the result is not
true. For example, it is well-known (and was first observed by John
Reeve~\cite{Reeve}) that there are simplices such as
\[
 P_k = \conv \left( 
  \begin{smallmatrix}
    0&1&0&1 \\
    0&0&1&1 \\
    0&0&0&k
  \end{smallmatrix} \right)
  \]
with only $4$ lattice points but unbounded volume.
In particular, this  shows that 
the number of lattice points of a lattice polytope does not
give a bound on its volume. 
\end{example}

\medskip

On the classification side, most of the known results concern toric
Fano varieties. Equivalently, on the polyhedral side the
classifications deal with polytopes for which the primitive ray
generators of the normal fan are the vertices of a convex polytope.
In dimension two, $\Q$-Gorenstein toric Fano surfaces are known for
Gorenstein index $\le 17$ \cite{KKN10}. In dimension three, the finite
list of canonical toric Fano varieties was obtained by A.~Kasprzyk
\cite{Kas06}.  We refer the interested reader to the Graded Ring
Database \href{http://grdb.lboro.ac.uk}{\texttt{grdb.lboro.ac.uk}} for
these and other classification results. Gorenstein toric Fano
varieties, corresponding to so-called reflexive polytopes
\cite{Bat94}, are completely classified in dimension $\le 4$
\cite{KS98,KS00}. Toric Fano manifolds are classified up to dimension
$8$ \cite{Bat99, Sat00, KN09, Oeb07}; recently, B.~Lorenz computed
dimension~$9$. The complete list of the corresponding smooth reflexive
polytopes can be found in the database at
\href{http://polymake.org}{\texttt{polymake.org}}.

Higher-dimensional classification results of toric varieties are only
known in two cases:
in the Gorenstein Fano case under strong symmetry assumptions
\cite{VK85,Ewald,Nil06} or if the Picard number of a toric manifold is
at most $3$, i.e., the $d$-dimensional fan has at most $d+3$ rays, in
which case the variety is automatically projective \cite{KS91, Bat91}.

\subsection*{Acknowledgements}
{This project started during the AIM workshop
``Combinatorial challenges in toric varieties''. Bernd Sturmfels asked
the finiteness question, and the proof was worked out by the present
authors, assisted by Sandra Di Rocco, Alicia Dickenstein,
Diane Maclagan and Greg Smith. Benjamin Lorenz carried out the
classification
in his
Diploma thesis~\cite{LorenzDiplom}.
Work of Haase, Nill, and Lorenz supported by Emmy Noether and
Heisenberg grants HA4383/1, HA4383/4 of the German Research Society
(DFG). Work of Nill also supported by NSF grant DMS 1203162. 
Work of Hering supported by NSF grant DMS 1001859. Work of Paffenholz is supported by the Priority Program 1489 of the German Research Council (DFG). Work of Santos supported by the Spanish Ministry of Science
through grants MTM2011-22792 and CSD2006-00032 (i-MATH)}

\section{Polarized toric varieties and lattice polytopes.}
\label{preliminaries}

In this section we introduce notation and recall some basic 
facts about toric varieties. For more details 
we refer to~\cite[\S2.3]{CoxLittleSchenckToricBook}
or~\cite[Section~3.4]{Fulton}. 

\subsection{Lattice Polytopes}

Let $N\cong \Z^d$ be a lattice with dual
lattice $M = \Hom (N, \Z)$ and associated vector spaces $N_{\R} : =
N\otimes _{\Z}\R$ and $M_{\R} : = M\otimes _{\Z}\R$. A \emph{lattice
  polytope} $P\subseteq M_{\R}$ is the convex hull of a finite number
$u_1, u_2, \ldots, u_r$ of points in $M$. Any lattice polytope is the
intersection of finitely many affine half spaces with primitive normal
vectors $v_1, \ldots, v_s$ in $N$:
\begin{align*}
  P \ =\ \conv(m_1, \ldots, m_r)\ =\ \{u\in M_{\R}\mid \langle
  v_j, u\rangle\ge -\alpha_j\,, \; 1\le j\le s\}
\end{align*}
for integral $\alpha_j$'s. A \emph{face} of $P$ is the intersection of
$P$ with an affine hyperplane $H$ such that $P$ is completely
contained in one of the affine half spaces defined by $H$. Faces of a
lattice polytope are lattice polytopes themselves.

For a vertex $u$ of $P$, let $T_uP := \cone ( u'-u \mid u'\in P )$ be
the (inner) \emph{tangent cone} to $P$ at $u$.  It is dual to the
(inner) \emph{normal cone} $\sigma(P,u):=\{v\in N_{\R}: \langle
u'-u,v\rangle \ge 0\ \forall u'\in P\}$ of $P$ at $u$.  The normal
cones of the different vertices of $P$ together with their faces form
a polyhedral decomposition of $N_{\R}$ called the \emph{normal fan} of
$P$.

For a subset $S$ of $M_{\R}$, let $\aff(S)$ denote the affine span of
$S$.  We say that two lattice polytopes $P \subset M_{\R}$ and $P'
\subset M'_{\R}$ for lattices $M$ and $M'$ are \emph{integrally
  equivalent} if there is a lattice preserving affine map $\aff P \to
\aff P'$ that maps $M \cap \aff P$ bijectively to $M' \cap \aff P'$
and $P$ to $P'$.  Up to this integral equivalence, we can (and will)
always assume that our polytope $P$ is full dimensional, i.e. $\aff
P=M_{\R}$.

Let $e_1, \ldots, e_d$ be any basis of the lattice $M$. The
\emph{normalized volume} $V(P)$ is the volume that assigns $1$ to the
simplex $\conv(0,e_1, \ldots, e_d)$. In dimension $2$ Pick's
formula~\cite{pick} relates the normalized volume $V$ with the number
$i$ of interior lattice points and the number $b$ of boundary lattice
points via
\begin{align}
  V+2\ =\ 2i+b\,.\label{eq:pick}
\end{align}

\subsection{Line bundles and polytopes}
Let $k$ be an arbitrary field and 
let $\Sigma$ be a complete rational fan of dimension $d$ in $N_{\R}$.
Let $X = X(\Sigma)$ be the associated \emph{toric variety}, a normal
equivariant compactification of the algebraic torus $T\cong
(k^*)^d$. The dual lattice $M$ is naturally isomorphic to the
character lattice of $T$. Assume that $X$ is \emph{projective}
(equivalently, that $\Sigma$ is the \emph{normal fan} of a polytope),
and let $\L$ be an ample line bundle on $X$. The polarized toric
variety $(X, \L)$ corresponds to a lattice polytope $P \subseteq
M_{\R}$ of dimension $d$ with its normal fan equal
to $\Sigma$.  Moreover, we have an isomorphism
\[ H^0(X, \L) \cong \bigoplus _{u\in P\cap M} k \chi^u, \]
where $\chi ^u\colon  T \to k^*, (t_1, \ldots, t_d) \mapsto 
t_1^{u_1}\cdots t_d^{u_d}$ is the character corresponding to 
$u\in M$. 

A linear series $W\subseteq H^0(X,\L)$ induces a rational map $X
\dashrightarrow \P(W)$, which is equivariant if and only if $W$ is
torus invariant, that is, $W\cong \oplus _{u\in S} k\chi^u$ for some
$S\subseteq P\cap M$.  Letting $S=\{u_1,\dots,u_{|S|}\}$, the induced
map is given by $x \mapsto [\chi^{u_1}(x)\colon \cdots\colon
\chi^{u_{|S|}}(x)]$.  The degree of this map turns out to be the
\emph{normalized volume} of $\conv(S)$ -- the volume measured in
volumes of unimodular simplices.
The map is induced by a \emph{complete linear series} $W$
if and only if $W = H^0(X,\L)$, that is, $S=P\cap M$.  See
\cite[\S6]{CoxLittleSchenckToricBook}.

\begin{figure}[htp]
  \begin{minipage}[b]{.42\textwidth}
    \centering
    \begin{tikzpicture}[scale=.6, y=-1cm]
      \definecolor{fillColor}{gray}{0.65098}
      \path[draw=black,thick,fill=fillColor] (12.7,5.08) -- (11.43,5.08) -- (11.43,3.81) -- (12.7,3.81);
      
      \fill[draw=black] (11.43,5.08) circle (0.15875cm);
      \fill[draw=black] (11.43,3.81) circle (0.15875cm);
      \fill[draw=black] (12.7,3.81) circle (0.15875cm);
      \fill[draw=black] (12.7,5.08) circle (0.15875cm);
      \draw[thick,black] (12.7,5.08) -- (12.7,3.81);
      \path (11.43,3.4925) node[text=black,anchor=base east] {
$u_2$};
      \path (12.7,3.4925) node[text=black,anchor=base west] {
$u_3$};
      \path (12.7,5.87375) node[text=black,anchor=base west] {
$u_1$};
      \path (11.43,5.87375) node[text=black,anchor=base east]{
$u_0$};
    \end{tikzpicture}%
    \setcaptionwidth{\textwidth} 
      \caption{The Segre embedding $\P^1 \times \P^1 \hookrightarrow
        \P^3$ via $\OO(1,1)$}
  \end{minipage}
  \quad \qquad
  \begin{minipage}[b]{.42\textwidth}
\centering
\begin{tikzpicture}[scale=.6, y=-1cm]

\definecolor{fillColor}{gray}{0.65098}
\path[draw=black,thick,fill=fillColor] (13.97,5.08) -- (11.43,5.08) -- (11.43,2.54) -- cycle;

\fill[draw=black] (12.7,3.81) circle (0.15875cm);
\fill[draw=black] (12.7,5.08) circle (0.15875cm);
\fill[draw=black] (13.97,5.08) circle (0.15875cm);
\fill[draw=black] (11.43,2.54) circle (0.15875cm);
\fill[draw=black] (11.43,3.81) circle (0.15875cm);
\fill[draw=black] (11.43,5.08) circle (0.15875cm);
\path (11.1125,5.23875) node[text=black,anchor=base east] {
$u_0$};
\path (11.1125,3.96875) node[text=black,anchor=base east] {
$u_3$};
\path (11.1125,2.69875) node[text=black,anchor=base east] {
$u_5$};
\path (12.85875,3.4925) node[text=black,anchor=base west] {
$u_4$};
\path (14.12875,4.7625) node[text=black,anchor=base west] {
$u_2$};
\path (12.7,6) node[text=black,anchor=base] {
$u_1$};
\end{tikzpicture}%

    \setcaptionwidth{\textwidth} 
    \caption{The Veronese embedding $\P^2 \hookrightarrow
      \P^5$ via $\OO(2)$}
  \end{minipage}
\end{figure}

If $P$ and $P'$ are integrally equivalent, and if $(X, \L)$ and $(X',
\L')$ are the corresponding polarized toric varieties, then there
exists a torus equivariant isomorphism $\phi \colon X \to X'$ such
that $\phi^* \L' \cong \L$.

\subsection{Singularities and cones}
\label{sec:cones}

Let $\L$ be an ample line bundle on the toric variety $X(\Sigma)$ with
corresponding lattice polytope $P \subseteq M_{\R}$. Then $X(\Sigma)$
is covered by torus invariant affine pieces $U_u$ which correspond to
the vertices $u$ of $P$. 

For each tangent cone, the semigroup $T_uP \cap M$ is finitely
generated. Its unique minimal set of generators $\hilb(T_uP)$ 
is called the \emph{Hilbert basis} of the
cone~\cite[Proposition~1.2.22]{CoxLittleSchenckToricBook}.
The coordinate ring of the affine variety $U_u$ is the semigroup ring
$k[U_u] \ = \ k[T_uP \cap M]\,.$

The line bundle $\L$ is called \emph{very ample} if its global 
sections induce an embedding into projective space. 
The combinatorial condition for $\L$ to be very ample 
is that for every vertex $u$ of $P$, the shifted polytope $P-u$
contains the Hilbert basis, i.e.,  $\hilb(T_uP) \subseteq P-u$,
see~\cite[Section~3.4]{Fulton}. We call $P$ \emph{very ample} if this
happens.

\begin{example}[Example~\ref{singex} continued]
\label{singex2}
The line bundle corresponding to Reeve's simplex $P_k$ is not very
ample. The line bundle corresponding to $2P_k$ is normally generated,
so in particular it is very ample (see \cite[Theorem
1.3.3]{BrunsGubeladzeTrung97} or \cite{NO}). It induces an embedding
into $\P^{k+8}$.
\end{example}

\subsubsection{$\Q$-Gorenstein cones}
\label{def-multiplicity}
Let $\sigma \subset N_{\R}\cong \R^d$ be a pointed rational $d$-cone
with primitive generators $v_1,\ldots, v_r$. We call $\sigma$
\emph{$\Q$-Gorenstein} if the $v_i$ lie in an affine hyperplane in
$N_{\R}$. That is, if there is a linear functional on $\sigma$ which
takes the value $1$ on all $v_i$. This functional is called
\emph{height} and denoted $\height_\sigma \in M_{\R}$ . The
\emph{index} of $\sigma$ is the smallest $k \in \Z_{>0}$ so that
$k\cdot\height_\sigma \in M$.
We call $\sigma$ \emph{Gorenstein} if this index is equal to $1$.

These notions agree with the notions ($\Q$-)Gorenstein and index 
for the toric singularity associated with $\sigma$.
We define the \emph{multiplicity} $\mult(\sigma)$ as the normalized
volume of the \emph{nib} of $\sigma$
\[
\nib(\sigma) := \conv ( 0, v_1,\ldots, v_r ) = \{ x \in \sigma \mid
\langle \height_\sigma, x \rangle \le 1 \}
\]
which equals the product of the index with the normalized volume of
$\conv ( v_1,\ldots, v_r )$.
Observe that every simplicial cone is $\Q$-Gorenstein, and its
multiplicity equals $\det( v_1,\ldots, v_r)$.

Let $P$ be a lattice polytope with $\Q$-Gorenstein normal fan.
We define the \emph{multiplicity} of $P$ to be 
\[ \mult(P) = \max _{u} \mult(\sigma(P,u)),\] 
the maximal multiplicity of a normal cone to $P$. 

Note that for a projective toric variety $X$, the multiplicity does 
not depend on the polarization, so we can define 
the multiplicity $\mult(X) = \mult(P)$, where 
$P$ is a lattice polytope corresponding to an ample line bundle 
on $X$. 

\subsubsection{Simplicial cones}
The toric singularity $U_u$ is $\Q$-factorial if the tangent cone
$T_uP$ of $P$ at $u$ is simplicial, that is, it is generated by
a linearly independent set $\{v_1, \ldots, v_d\}$ of primitive
vectors. In this case, the singularity $U_u$ is a quotient
$k^d/G$ of affine space by a finite abelian group, and the
multiplicity is the cardinality of that group.
The \emph{box} of $T_uP$ is the half open parallelepiped
$$ \Box(T_uP) := \left\{ \sum_{i=1}^d \lambda_iv_i \mid \lambda_i \in
  [0,1) \text{ for } i = 1, \ldots, d \right\}\,,$$
and a \emph{box point} is one of the $\mult(T_uP)$ many lattice points
in $\Box(T_uP)$. Every Hilbert basis element that is not one of the
generators of $T_uP$ is a box point, and has smaller height than
$d$. In particular, we have $\hilb(T_uP)\setminus\{v_1, \ldots, v_d\}
\subset \Box(T_uP)$.

A cone is called \emph{unimodular} if its primitive minimal generators
form a lattice basis. 
Unimodularity is equivalent to having multiplicity $1$.  We call a
lattice polytope $P$ \emph{smooth} if every cone in its normal fan is
unimodular.

A lattice polytope is smooth if and only if the associated
projective toric variety $X$ is smooth (see for example~\cite[Section
2.1]{Fulton}).  Moreover, every ample line bundle on a smooth toric
variety is very ample.

\section{Finiteness Theorems}
\label{sec:theorems}

When we bound the number of lattice points, we arrive fairly quickly
at the desired finiteness result for smooth polytopes (see
Section~\ref{sec:fewpolytopes}). The case of simple and very ample
polytopes is treated in Section~\ref{sec:very_ample}. Finally, in
Section~\ref{sec:2dim}, we show that for polytopes with restricted
normal cones it suffices to bound the number of lattice points on the
edges.

\subsection{Few polytopes with $\nbr$ lattice points}
\label{sec:fewpolytopes}

Our finiteness theorems are based on the analysis of what happens in dimension two and then applying the following Lemma. 

\begin{lemma}\label{lem:2dtoalld}
Let $\nbr> d\geq 2$ be positive integers, let $\mathcal{F}$ be a finite family of $d$-dimensional lattice cones and let 
$\mathcal{P}$ be a finite family of lattice polygons. 
There are, up to integral equivalence, finitely many lattice
$d$-polytopes with less than $n$ vertices such that every
2-dimensional face is integrally equivalent to a polygon from
$\mathcal{P}$ and every normal cone is integrally equivalent to a cone
from $\mathcal{F}$.
\end{lemma}

\begin{proof}
  We first observe that there is a finite number of combinatorial
  types of fans $\Sigma$ with $\le \nbr$ maximal cones. Here, the
  \emph{combinatorial type} is given by the set of faces, partially
  ordered by inclusion.
  Once the combinatorial type of $\Sigma$ is fixed, there are  
  only finitely many choices to assign an element $[\sigma_u]$ 
  of $\mathcal F$ to a maximal cone of $\Sigma$ and to embed the face
  poset of $\sigma_u$ into the face  poset of $\Sigma$ (if possible at
  all). So, we only need to prove finiteness of the number of
  polytopes $P$ with a fixed combinatorial type so that at every
  vertex $u$ the edges containing $u$ are assigned facets of
  $[\sigma_u]$ and such that every two-dimensional face is integrally
  equivalent to a polygon in $\mathcal{P}$.
  There are only finitely many ways to embed the combinatorial type of
  a polygon from $\mathcal P$ into the combinatorial type of a
  $2$-face of $P$. We claim that these choices actually determine $P$
  up to equivalence.
  
  To this end, fix a vertex $u$ of $P$ and an element $\sigma_u$ of
  the equivalence class in  $\mathcal F$ that we assigned to $u$.
  This determines all $2$-dimensional faces of $P$ incident to $u$.
  In particular, if $u'$ is another vertex of $P$ adjacent to $u$,
  $u'$ together with all edges which are incident to $u'$ and
  contained in a common $2$-face with $u$ are determined. The
  directions of these edges, together with the edge $uu'$ span $\R^d$
  as a vector space.
  They thus pin down the normal cone $\sigma_{u'}$ in its class.

  In summary, fixing a vertex and its normal cone also fixes all
  adjacent vertices and their normal cones. As the vertex-edge graph
  of $P$ is connected, this determines $P$.
\end{proof}

\begin{theorem}\label{thm:points}
Let $\nbr> d\geq 2$ be positive integers, and let $\mathcal{F}$ be a
finite family of $d$-dimensional lattice cones.
There are, up to integral equivalence, finitely many lattice
$d$-polytopes with at most $n$ lattice points such that every normal
cone is equivalent to a cone from $\mathcal{F}$.
\end{theorem}

\begin{proof}
In dimension two the statement follows from
Theorem~\ref{thm:LZ}. Indeed, Pick's formula~\eqref{eq:pick} implies
that the volume $V$ and the number of lattice points $\nbr$ of a
lattice polygon bound each other:
  \[
  \nbr\le V+2\le 2\nbr-3.
  \]
  Then Lemma~\ref{lem:2dtoalld} implies the theorem. 
\end{proof}

By taking $\mathcal F$ to consist of a single element, the unimodular
cone, Theorem~\ref{thm:points} implies the following weak version of
Theorem~\ref{thm:polytopes}:

\begin{corollary}
\label{coro:smoothembeddings}
  Let $\nbr$ be a nonnegative integer. Then, there are only finitely
  many smooth lattice polytopes with $\nbr$ lattice points.
\end{corollary}

\begin{example}\label{Hirzebruch}
Corollary~\ref{coro:smoothembeddings} does not imply that there are only
finitely many projective torus equivariant embeddings into a fixed
projective space. If we don't require the linear series to be complete,
Figure~\ref{hirzebruchpic} shows how to embed an
arbitrary Hirzebruch surface torically into $\P^5$. 
\end{example}

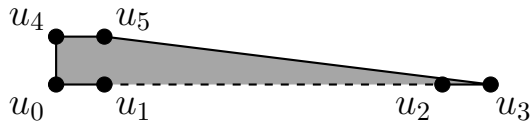
\begin{figure}[htp] 
  \begin{center}
\begin{tikzpicture}[scale=.5, y=-1cm]

\definecolor{fillColor}{gray}{0.65098}
\path[fill=fillColor] (11.43,5.08) -- (11.43,3.81) -- (12.7,3.81) -- (22.86,5.08);

\draw[thick,black] (12.7,5.08) -- (11.43,5.08) -- (11.43,3.81) -- (12.7,3.81);

\fill[draw=black] (11.43,5.08) circle (0.2cm);
\fill[draw=black] (11.43,3.81) circle (0.2cm);
\fill[draw=black] (12.7,5.08) circle (0.2cm);
\fill[draw=black] (12.7,3.81) circle (0.2cm);
\fill[draw=black] (21.59,5.08) circle (0.2cm);
\fill[draw=black] (22.86,5.08) circle (0.2cm);
\draw[thick,black] (21.59,5.08) -- (22.86,5.08) -- (12.7,3.81);
\draw[thick,dashed,black] (12.7,5.08) -- (21.59,5.08);
\path (11.43,5.87375) node[text=black,anchor=base east] {\large{}$u_0$};
\path (12.7,5.87375) node[text=black,anchor=base west] {\large{}$u_1$};
\path (21.59,5.87375) node[text=black,anchor=base east] {\large{}$u_2$};
\path (22.70125,5.87375) node[text=black,anchor=base west] {\large{}$u_3$};
\path (11.43,3.4925) node[text=black,anchor=base east] {\large{}$u_4$};
\path (12.7,3.4925) node[text=black,anchor=base west] {\large{}$u_5$};

\end{tikzpicture}%
  \end{center}
  \caption{Hirzebruch surface $X(\F_a) \hookrightarrow \P^5$}
  \label{hirzebruchpic}
\end{figure}

\begin{corollary} \label{cor:Gorenstein}
  For nonnegative integers $m$ and $\nbr$, there are only finitely
  many lattice polytopes with $\Q$-Gorenstein normal cones of
  multiplicity bounded by $m$ and with $\nbr$ lattice points.
\end{corollary}

\begin{proof}
  Applying Theorem~\ref{thm:LZ} to the convex hull of $0$ and the
  primitive generators of a $\Q$-Gorenstein cone, we see that the
  family of $\Q$-Gorenstein cones with multiplicity $\le m$ contains
  only finitely many equivalence classes. Now apply
  Theorem~\ref{thm:points}.
\end{proof}

We consider two morphisms to $\P^\nbr$ the same if they differ by an
automorphism of $\P^\nbr$.  Using the dictionary between toric
morphisms and lattice polytopes, Corollary \ref{cor:Gorenstein}
implies the following corollary.

\begin{corollary}\label{morphisms}
Let $\nbr$ and $m$ be nonnegative integers. There are 
fi\-ni\-te\-ly many morphisms from some 
$\Q$-Gorenstein toric variety $X$
with $\mult(X) \leq m$
to $\P^\nbr$ that are induced by a complete linear series. 
\end{corollary}

Example~\ref{singex} shows that the assumption that the multiplicities
are bounded in Corollaries~\ref{cor:Gorenstein} and~\ref{morphisms} is
needed.

\subsection{Simple, very ample polytopes with $\nbr$ lattice points}
\label{sec:very_ample}

In this section we will show that when $P$ is simple and very ample, then 
the multiplicity of $P$ is bounded and so the corresponding assumption in 
Corollary~\ref{cor:Gorenstein} comes for free. 

In order to deduce Theorem~\ref{thm:Qfactorialembeddings} from
Corollary~\ref{morphisms}, we need another lemma.

\begin{lemma}\label{lemma:veryample}
  For nonnegative integers $\nbr$ and $d$ there are
  only finitely many $\Q$-Gorenstein cones $\sigma \subset \R^d$ so
  that
  \begin{equation} \tag{$\ast$} \label{eq:veryampleBound}
    \# \hilb(\sigma) \ + \ \#(\nib(\sigma) \cap \Z^d) \ \le \ \nbr \,.
  \end{equation}
\end{lemma}

Observe that bounding $\# \hilb(\sigma)$ or $\#(\nib(\sigma) \cap
\Z^d)$ alone is not enough.
Examples~\ref{egBruns} and~\ref{singex} show infinitely many cones
with bounded $\#(\nib(\sigma) \cap \Z^d)$, and the following cones
have only three Hilbert basis elements, and multiplicity $2a$:
\[
C_a:=\cone\{(1,a),(-1,a)\}\subseteq \R^2, \quad
\hilb(C_a)=\{(1,a),(-1,a),(0,1)\}.
\]

\begin{proof}[Proof of Lemma~\ref{lemma:veryample}]
  We will show by induction on $d$ that \eqref{eq:veryampleBound}
  implies that  
  $\mult(\sigma)$ is bounded. 
  Then  Theorem~\ref{thm:LZ}  implies that there are 
  only finitely many choices for $\sigma$.

  For $d=1$ there is only one
  cone. For $d=2$, Pick's formula~\eqref{eq:pick} tells us that
  $\mult(\sigma) \leq  2\#(\nib(\sigma) \cap \Z^2) - 5$.
  So let us assume that the lemma is true for $d-1$. 
  Because of Corollary~\ref{cor:LZ}, we can assume that
  $\nib(\sigma)$ has no interior lattice points. This implies that all
  interior Hilbert basis elements of $\sigma$ have height $\ge 1$.
  By induction, there is a minimal height $\epsilon(d-1,\nbr) > 0$,
  depending only on $d-1$ and $\nbr$, of a Hilbert basis element in
  the boundary of $\sigma$. Let $\epsilon = \min \{\epsilon(d-1,\nbr),1\}$. 

  Triangulate $\sigma = \cup_{i=1}^r \sigma_i$ into simplicial cones
  using only rays of $\sigma$. Every Hilbert basis element of $\sigma$
  is a box point of one of the $\sigma_i$. As $\sigma$ has at most $\nbr$
  rays, every box point belongs to less than $\binom{\nbr}{d}$ of the
  $\sigma_i$.
  
  Now, every box point of every $\sigma_i$ has a representation
  $\sum_{v \in \hilb(\sigma)} a_vv$ with $a_v \in \Z_{\ge0}$. On the
  other hand, any box point has height $< d$, so that in the above
  representation we must have $\epsilon \cdot \sum_{v \in \hilb(\sigma)} a_v <
  d$ which leaves at most $\left( \begin{smallmatrix} \nbr + \lfloor
      d/\epsilon \rfloor \\ \nbr \end{smallmatrix} \right)$
  possibilities for the coefficients $a_v$.
  In other words, 
  $$\mult(\sigma) = \sum_{i=1}^r \# \Box(\sigma_i) < 
  \binom{\nbr}{d} \cdot \# \left( \bigcup_{i=1}^r \Box(\sigma_i) \right) <
  \binom{\nbr}{d} \binom{\nbr + \lfloor d/\epsilon \rfloor}{\nbr} \,.$$
\end{proof}

The following statement is equivalent to
Theorem~\ref{thm:Qfactorialembeddings}.

\begin{theorem}\label{thm:veryample}
Let $\nbr$ be a nonnegative integer. Then there exist only 
finitely many simple and very ample polytopes 
with $\nbr$ lattice points. 
\end{theorem}

\begin{proof}
Since $P$ is simple, every tangent cone to $P$ is $\Q$-Gorenstein.
Moreover, since $P$ is very ample, a translate of the Hilbert basis
for each tangent cone is a subset of the lattice points of $P$. Since
$P$ has $\nbr$ lattice points, it follows from Lemma
\ref{lemma:veryample} that there are only finitely many equivalence
classes of tangent cones. So there are only finitely many equivalence
classes of normal cones. Now the claim follows from Theorem
\ref{thm:points}.
\end{proof}

The following example shows that we need to assume that $P$ is simple
(resp.,~that $X$ is $\Q$-factorial) in Theorem~\ref{thm:veryample}
(resp.,~Theorem~\ref{thm:Qfactorialembeddings}).

\begin{example}\label{egBruns}
  In ~\cite[p.2290]{MFO07} Winfried Bruns gives an example of a very
  ample divisor on a toric $3$-fold whose complete linear series does
  not yield a projectively normal embedding. This example generalizes
  to a family of very ample polytopes
  $$ Q_k := \conv \left(
    \begin{smallmatrix}
      0&1&0&0&1&0&1&1 \\
      0&0&1&0&0&1&1&1 \\
      0&0&0&1&1&1&k&k+1
    \end{smallmatrix} \right)$$
  with $8$ lattice points but unbounded volume. Observe that these
  polytopes have a Gorenstein normal fan with 
  $\mult(Q_k)=k+1$. However, the tangent cone $T_{(0,1,0)}(Q_k) =
  \cone((0,-1,0),(0,0,1),(1,-1,0),(1,0,k))$ is not $\Q$-Gorenstein for
  $k\ge 2$.
\end{example}

\subsection{Polytopes with $\nbr$ lattice points on their edges}
\label{sec:2dim}

The proof of Theorem~\ref{thm:points} and, hence, those of
Corollaries~\ref{coro:smoothembeddings}
and~\ref{cor:Gorenstein}/\ref{morphisms}, were based on Pick's
formula~\eqref{eq:pick}, which allowed us to bound the number of
equivalence classes of polygons with a given number of lattice
points. We now show that bounding the number of lattice points
\emph{along the edges} of the polygons is enough, if we also put a
bound on the multiplicity of the cones.
The following example shows that bounding the multiplicity is
necessary.

\begin{example}\label{ex:multnec} 
The polygons $ P_{pq} = \conv((-1,-1),(p,0),(0,q))$ for $p$ and $q$
relatively prime positive integers form an infinite family of polygons
having only 3 lattice points on their edges.
\end{example}

We call a lattice polygon $P$ a \emph{$(m,\nbr)$-polygon} if $P$ has
at most $\nbr$ lattice points on the boundary and $\mult(P) \le m$.
Then our most general finiteness result for polygons is the following
theorem.
\begin{theorem}
  \label{thm:mainforpolygons}
  Let $m$ and $n$ be positive integers.
  There are only finitely many integral equivalence classes of 
  $(m,\nbr)$-polygons.
\end{theorem}

Before proving Theorem~\ref{thm:mainforpolygons} in Section~\ref{sec:finpoly},
the following strong versions of Theorem~\ref{thm:points}
and Corollaries~\ref{coro:smoothembeddings} and~\ref{cor:Gorenstein}
are derived. 
\begin{theorem}\label{thm:points2}
  Let $\nbr > d\geq 2$ be positive integers, and let $\mathcal{F}$ be
  a finite family of $d$-dimensional lattice cones.
  There are, up to integral equivalence, finitely many lattice
  $d$-polytopes with at most $\nbr$ lattice points on edges and such
  that every normal cone is equivalent to a cone from $\mathcal{F}$.
\end{theorem}

\begin{proof}
Let $P$ be a $d$-polytope such that every normal cone is in
$\mathcal{F}$. This implies that the normal cones to every
two-dimensional face of $P$ are contained in a finite family of
two-dimensional cones $\mathcal{F'}$. In particular, each such
two-dimensional face has bounded multiplicity. Since the number
of lattice points on the edges of a face of a polytope $P$ are bounded
by the number of lattice points on the edges of $P$, it then follows
from Theorem~\ref{thm:mainforpolygons} that the set $\mathcal{P}$ of
polygons that can occur as two-dimensional faces of a polytope $P$
satisfying the assumptions of the theorem is finite. 
Now the claim follows from Lemma~\ref{lem:2dtoalld}. 
\end{proof}

We can now apply this to prove our main Theorem. 
\begin{proof}[Proof of Theorem~\ref{thm:smooth}/~\ref{thm:polytopes}]
Fix $\nbr$. Then any lattice polytope with less than $\nbr$ lattice
points on its edges has less than $\nbr$ vertices, so $d=\dim(P) \leq
\nbr-1$. So it is enough to show that there are finitely many smooth
lattice polytopes of dimension $d$ with  less than $\nbr$ lattice
points on their edges. When $P$ is smooth, then every normal cone to
$P$ is unimodular, so this follows from applying
Theorem~\ref{thm:points2} to $\nbr\geq n'>d$ and $\mathcal{F}$
consisting of the unimodular cone of dimension $d$.
\end{proof}

Theorem~\ref{thm:smooth} implies the following. We are not aware of a
more direct proof of this statement.
\begin{corollary}
For smooth lattice polytopes, bounding the number of lattice points on
the edges bounds the number of total lattice points. In other words,
for a line bundle on a smooth polarized toric variety $(X,L)$ bounding
$d(L)$ bounds  $h^0(X,L)$.
\end{corollary}

As lattice polygons are always $\Q$-Go\-ren\-stein and very ample,
Example~\ref{ex:multnec} shows that the statement of
Theorem~\ref{thm:smooth} does not hold when we only assume that $P$ is
simple and very ample. In fact, in the proof of
Theorem~\ref{thm:veryample}, we used the total number of lattice
points to bound the multiplicity. If we assume in addition that the
multiplicity is bounded, we obtain the following.

\begin{corollary} \label{cor:Gorenstein2}
  For nonnegative integers $m$ and $\nbr$, there are only finitely
  many lattice polytopes with $\Q$-Gorenstein normal cones of
  multiplicity bounded by $m$ and with $\nbr$ lattice points on
  edges.
\end{corollary}
\begin{proof}
The first part is like the proof of Corollary~\ref{cor:Gorenstein},
but then apply Theorem~\ref{thm:points2}.
\end{proof}

\subsubsection{Finitely many polygons}\label{sec:finpoly}

In this section we prove Theorem~\ref{thm:mainforpolygons}, arguing on
the normal fan of a polygon. A $2$-dimensional polyhedral fan $\Sigma$
is a \emph{$[k,n]$-fan} if it is complete, has $k'\le k$ rays
(one-dimensional faces) with primitive generators $v_1, \ldots,
v_{k'}\in N$ such that there are non-negative integers $\lambda_1,
\ldots, \lambda_{k'}$ with $\sum \lambda_i v_i=0$, $\sum\lambda_i\le
n$ and $\mathrm{span}(v_j\mid \lambda_j\ne 0) =N_{\R}$. The last
condition means that there exists a polygon whose normal fan is
refined by  $\Sigma$ with at most $n$ lattice points on its boundary.

\begin{lemma} \label{lemma:dn-polygon-to-ndn-fan}
  Let $P$ be an $(m,n)$-polygon with normal fan $\Sigma$. Then there
  is a unimodular $[nm,n]$-fan $\overline \Sigma$ refining $\Sigma$.
\end{lemma}

\begin{proof}
  The minimal unimodular subdivision $\overline \Sigma$ of $\Sigma$
  (as discussed in detail in \cite[\S10.2]{CoxLittleSchenckToricBook})
  introduces less than $m$ new rays for each cone of multiplicity $\le
  m$. So $\overline \Sigma$ has at most $nm$ rays.

  Let $\lambda_v$ be the lattice length of the edge of $P$ dual to a
  ray $v$. In particular, $\lambda_v=0$ for the extra generators
  introduced in the refinement from $\Sigma$ to $\overline
  \Sigma$. This choice of coefficients certifies $\overline \Sigma$
  as an $[nm,n]$-fan by Minkowski's Theorem (cf.~\cite[Lemma
  4.9]{MR2206625}, \cite[p.~332]{Gruenbaum}).
\end{proof}

For what follows, we need a classification result for complete
two-dimensional unimodular fans
(cf.~\cite[Theorem~V.6.6]{Ewald} or \cite[Section 2.5]{Fulton}).
\begin{enumerate}
\item Any complete two-dimensional unimodular fan $\Sigma$ is
integrally equivalent to either the fan of $\P^2$,
which is generated by the three vectors $(1,0)$, $(0,1)$, and
$(-1,-1)$,
 or to a refinement of the fan
$\F_a$ of a \emph{Hirzebruch surface} for $a \ge 0$: $\F_a$ is
the complete fan with rays generated by $v_1=(1,0)$, $v_2=(0,1)$,
$v_3=(-1,0)$, and $v_4=(a,-1)$. 
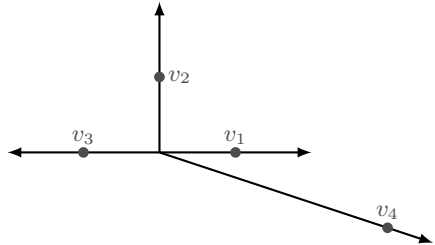
\begin{figure}[htb]
  \centering
{\scriptsize
\begin{tikzpicture}[scale=1.0]
\draw[thick,arrows={-latex}] (0,0) -- (-2,0);
\draw[thick,arrows={-latex}] (0,0) -- (2,0);
\draw[thick,arrows={-latex}] (0,0) -- (0,2);
\draw[thick,arrows={-latex}] (0,0) -- (3.6,-1.2);

\fill[black!70] (1,0) circle (2pt) node[above] {$v_1$};
\fill[black!70] (0,1) circle (2pt) node[right] {$v_2$};
\fill[black!70] (-1,0) circle (2pt) node[above] {$v_3$};
\fill[black!70] (3,-1) circle (2pt) node[above] {$v_4$};  
\end{tikzpicture}
}
  \caption{The fan $\F_a$ of the $a$-th Hirzebruch surface.}
\label{fig:hirzebruch-fan}
\end{figure}
\item The refinement from $\F_a$ to $\Sigma$ can be done introducing
  one ray at a time and in such a way that all intermediate fans are
  also unimodular. That is, in each refinement step a certain cone
  $\cone(v_1,v_2)$ with $|\det(v_1,v_2)|=1$ is subdivided into
  $\cone(v_1,v_3)$ and $\cone(v_3,v_2)$ with $v_3=v_1+v_2$. In
  polyhedral terms this is an example of a \emph{stellar
    subdivision}. In algebraic geometry terms, this corresponds to
  a blow-up at the fixed point corresponding to the cone
  $\cone(v_1,v_2)$.
\end{enumerate}

The following result bounds the parameter $a$ of the starting
Hirzebruch surface in terms of the parameters $m$ and $n$ of the fan:

\begin{lemma}\label{lem:Hirzebruch}
  Every  unimodular $[k,n]$-fan $\Sigma$ is equivalent to the fan of
  $\P^2$ or to a stellar subdivision of a Hirzebruch fan $\F_a$ with
  $0 \le a \le n-2$.
\end{lemma}

\begin{proof}
  Assume that $\Sigma$ is not the fan of $\P^2$ and let $a \ge 0$ be the
  minimal integer such that $\Sigma$ is a stellar subdivision of (a
  fan equivalent to) $\F_a$.
 In the case $a=0$ there is nothing to
  prove.
So let us     assume $a>0 $.
  Then the cones $\cone((0,1),(-1,0))$ and $\cone((-1,0),(a,-1))$ of
  $\F_a$ must be unsubdivided in $\Sigma$. Otherwise $\Sigma$ would
  contain the ray generated by either $(-1,1)$ or $(a-1,-1)$ and,
  hence, it would be a refinement of a fan equivalent to $\F_{a-1}$ as
  well.
  Hence, schematically, $\Sigma$ looks as in Figure~\ref{fig:hirz_subdiv}.
  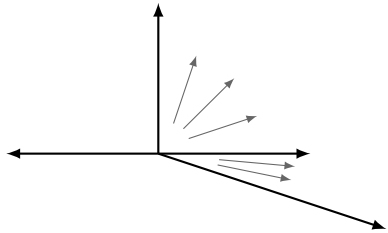
\begin{figure}
{\scriptsize
\begin{tikzpicture}[scale=1.0]
\draw[thick,arrows={-latex}] (0,0) -- (-2,0);
\draw[thick,arrows={-latex}] (0,0) -- (2,0);
\draw[arrows={-latex},black!60] (0.2,0.4) -- (0.5,1.3);
\draw[arrows={-latex},black!60] (0.33,0.33) -- (1.0,1.0);
\draw[arrows={-latex},black!60] (0.4,0.2) -- (1.3,0.5);
\draw[thick,arrows={-latex}] (0,0) -- (0,2);
\draw[arrows={-latex},black!60] (0.8,-0.08) -- (1.8,-0.17);
\draw[arrows={-latex},black!60] (0.78,-0.15) -- (1.75,-0.35);
\draw[thick,arrows={-latex}] (0,0) -- (3,-1.0); 
\end{tikzpicture}
}
    \caption{A schematic picture of the fan in the proof of
      Lemma~\ref{lem:Hirzebruch}}
     \label{fig:hirz_subdiv}
  \end{figure}
  Now consider the rays with non-zero coefficient $\lambda$ in the
  expression certifying that $\Sigma$ is a $[k,n]$-fan. Since they
  positively span $N_{\R}$, at least one of them must have negative
  first coordinate and at least one of them must have negative second
  coordinate. The discussion above implies that the only generator
  with negative first coordinate is $(-1,0)$, and that every generator
  with negative second coordinate has first coordinate $\ge a$. Hence,
  in order to have $\sum \lambda_v v=0$ we must have
  $\lambda_{(-1,0)}\ge a$. Hence, the sum of all coefficients, $n$, is
  at least $a+2$.
\end{proof}

\begin{proof}[Proof of Theorem~\ref{thm:mainforpolygons}]
As the $t$-th dilation of a unimodular triangle has $3t$ lattice
points on the boundary, there are at most $\lfloor n/3\rfloor$ of them
within our class. So, for the rest of the proof we bound the number of
polygons which are not dilations of unimodular triangles.  By
Lemma~\ref{lemma:dn-polygon-to-ndn-fan}, for any such polygon $P$
there exists a smooth $[nm,n]$-fan $\overline{\Sigma}$ refining the
dual fan of $P$.

By Lemma~\ref{lem:Hirzebruch}, $\overline{\Sigma}$ is obtained from
a  Hirzebruch surface $\mathbb{F}_a$ with $a\leq n-2$ by a sequence of
at most $nm-4$ unimodular stellar subdivisions. Since the number of
possible unimodular stellar subdivisions in a unimodular
$2$-dimensional fan with $i$ rays is finite (it actually equals $i$)
there is only a finite number of possibilities for the fan
$\overline{\Sigma}$, hence also for the polygon $P$.
\end{proof}

\subsubsection{How many polygons, and how big?}
We now look at the refinement process described above in more
detail in order to give estimates of how many polygons arise in
Theorem~\ref{thm:mainforpolygons} and what their maximum area is. We
do so only in the smooth case and obtain these results:

\begin{theorem}
\label{thm:fewpolygons}
Let $k\le n$ be positive integers. Then, the number of smooth $k$-gons
with $n$ boundary points is bounded above by
$4^k\binom{n}{k}\frac{n}{k}$.
\end{theorem}

\begin{theorem}
\label{thm:smallpolygons}
Let $k\le n$ be positive integers. Then, every smooth $k$-gon with $n$
boundary points has area bounded above by $\phi^{2k}n^2$, where
$\phi=1.618\dots$ is the golden ratio.
\end{theorem}

The starting point for the bound of Theorem~\ref{thm:fewpolygons} is
that the different unimodular refinements of a unimodular cone can be
recorded via binary trees. Remember that a binary tree is a rooted
tree in which every node other than the leaves has exactly two
children, labeled as ``left'' and ``right''. The number of different
binary trees with $k$ leaves is the Catalan number
$C_{k-1}=\frac{1}{k}\binom{2k-2}{ k-1}\le \frac{1}{k}
4^k$~\cite[Ex.\ 6.19(d)]{StanleyEnumerative2}.
If $\Sigma$ is a unimodular refinement of a unimodular $2$-dimensional
cone, we associate to $\Sigma$ the binary tree $T_\Sigma$ that has one
internal node for each ray introduced in the refinement process and
one leaf for each unimodular cone of $\Sigma$. See
Figure~\ref{fig:bintree} for an illustration. That is:

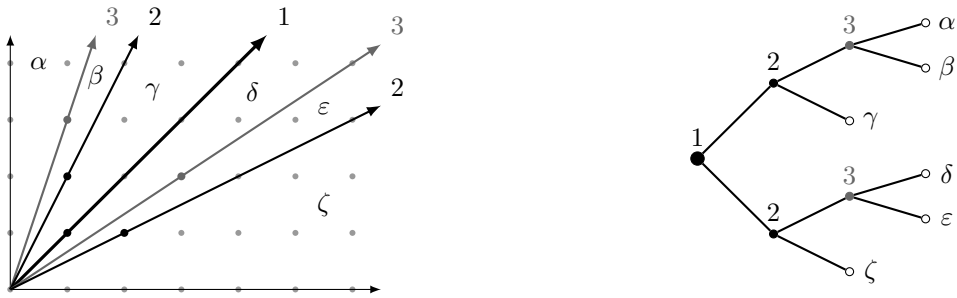
\begin{figure}[tbp]
      \begin{subfigure}[b]{0.52\textwidth}
         \begin{center}
{\footnotesize
\begin{tikzpicture}[scale=0.75]

\foreach \x in {0,...,6} {
\foreach \y in {0,...,4} {
\fill[black!40] (\x,\y) circle (1.5pt);
} }

\draw[black,arrows={-latex}] (0,0) -- (6.5,0);
\draw[black,arrows={-latex}] (0,0) -- (0,4.5);

\fill[black!60] (1,3) circle (2.0pt);
\draw[black!60,arrows={-latex},thick] (0,0) -- (1.500,4.5) node[above right] {$3$};
\fill[black] (1,2) circle (2.0pt);
\draw[black,arrows={-latex},thick] (0,0) -- (2.25,4.5) node[above right] {$2$};
\fill[black] (1,1) circle (2.0pt);
\draw[black,arrows={-latex},very thick] (0,0) -- (4.5,4.5) node[above right] {$1$};
\fill[black!60] (3,2) circle (2.0pt);
\draw[black!60,arrows={-latex},thick] (0,0) -- (6.5,4.333) node[above right] {$3$};
\fill[black] (2,1) circle (2.0pt);
\draw[black,arrows={-latex},thick] (0,0) -- (6.5,3.25) node[above right] {$2$};

\node at (0.50,4.00) {$\alpha$};
\node at (1.50,3.75) {$\beta$};
\node at (2.50,3.50) {$\gamma$};
\node at (4.25,3.50) {$\delta$};
\node at (5.50,3.20) {$\varepsilon$};
\node at (5.50,1.50) {$\zeta$};

\end{tikzpicture}
}
      \end{center}
      \end{subfigure}~
      \begin{subfigure}[b]{0.47\textwidth}
      \begin{center}
{\footnotesize
\begin{tikzpicture}[dotnode/.style={circle,draw=black,fill=none,inner sep=0pt,minimum size=3pt},node distance=8pt]

\node[dotnode,minimum size=5pt,fill=black,draw=black] (root) at (0,0) {};
\node[dotnode,fill=black,draw=black] (l) at (1.0,1.0) {};
\node[dotnode,fill=black,draw=black] (r) at (1.0,-1.0) {};
\node[dotnode,fill=black!60,draw=black!60] (ll) at (2,1.5) {};
\node[dotnode] (lr) at (2,0.5) {};
\node[dotnode] (lll) at (3,1.8) {};
\node[dotnode] (llr) at (3,1.2) {};
\node[dotnode,draw=black!60,fill=black!60] (rl) at (2,-0.5) {};
\node[dotnode] (rr) at (2,-1.5) {};
\node[dotnode] (rll) at (3,-0.2) {};
\node[dotnode] (rlr) at (3,-0.8) {};

\node[right of=lll] {$\alpha$};
\node[right of=llr] {$\beta$};
\node[right of=lr] {$\gamma$};
\node[right of=rll] {$\delta$};
\node[right of=rlr] {$\varepsilon$};
\node[right of=rr] {$\zeta$};

\node[above of=root] {1};
\node[above of=l] {2};
\node[above of=r] {2};
\node[above of=ll,black!60] {3};
\node[above of=rl,black!60] {3};

\draw[thick] (root) -- (l); 
\draw[thick] (root) -- (r); 
\draw[thick] (l) -- (ll); 
\draw[thick] (l) -- (lr); 
\draw[thick] (r) -- (rl); 
\draw[thick] (r) -- (rr); 
\draw[thick] (ll) -- (lll); 
\draw[thick] (ll) -- (llr); 
\draw[thick] (rl) -- (rll); 
\draw[thick] (rl) -- (rlr); 

\end{tikzpicture}
}
      \end{center}
      \end{subfigure}
\caption{The binary tree corresponding to a unimodular refinement of a
  unimodular cone.\label{fig:bintree}}
\end{figure}

\begin{lemma}\label{trees}
  There is a bijection between stellar subdivisions of a unimodular
  cone with $k-1$ interior rays and binary trees with $k$ leaves.
\qed
\end{lemma}

With this we can prove Theorem~\ref{thm:fewpolygons}:

\begin{proof}[Proof of Theorem~\ref{thm:fewpolygons}]
Apart from the case of the fan of a unimodular triangle, we need to
count how many refinements there are of $\F_a$ with $k$ rays in total,
and then how many ways to choose the coefficients $\lambda_v$ in such
a way that $\sum \lambda_v v=0$ and $\sum \lambda_v=n$. To bound this
number, we combine the four binary trees that refine the four cones of
$\F_a$ into a single tree with $k$ leaves, as shown in
Figure~\ref{fig:combinedtree}. The number of ways of doing this is
clearly smaller than $C_{k-1}\le 4^k/k$.
Observe that we are
over-counting for several reasons: first, the trees we get are only
those that have no leaf at depth $1$. Second, in the case $a>0$ we
actually only need two binary trees, not four (put differently, the
trees labeled $\beta$ and $\gamma$ in Figure~\ref{fig:combinedtree}
are empty). Third, in the case $a=0$ there may be several copies of
$\F_0$ in the fan $\Sigma$, which means there are different binary
trees giving the same $\Sigma$.

We need to count the number of choices for the $\lambda$'s. We can
bound this by the number of ways of partitioning $n$ into $k$ positive
summands $\lambda_1+\lambda_2+\cdots+\lambda_n$, which equals
$\binom{n-1}{k-1}\le \binom{n}{ k}$. Again, this is an overcount,
because we do not care about the condition $\sum \lambda_v v=0$.
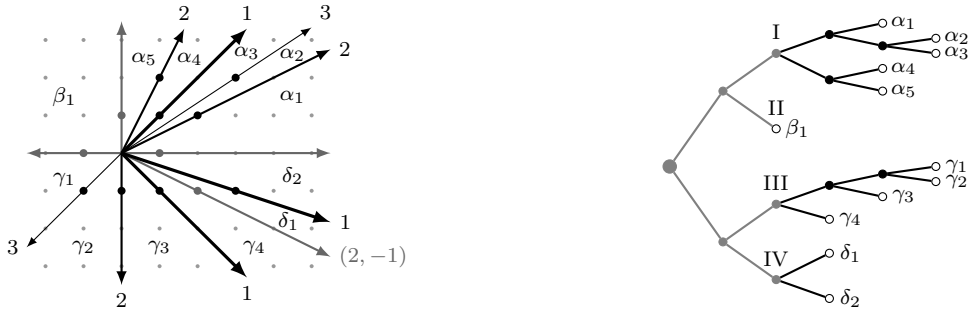
\begin{figure}[tbp]
   \begin{center}
      \begin{subfigure}[b]{0.52\textwidth}
         \begin{center}
{\tiny
\begin{tikzpicture}[scale=0.50]

\foreach \x in {-2,...,5} {
\foreach \y in {-3,...,3} {
\fill[black!40] (\x,\y) circle (1.5pt);
} }

\draw[thick,black!60,arrows={-latex}] (0,0) -- (5.5,0);
\draw[thick,black!60,arrows={-latex}] (0,0) -- (0,3.5);
\draw[thick,black!60,arrows={-latex}] (0,0) -- (-2.5,0);
\fill[thick,black!60] (2,-1) circle (3.0pt);
\draw[thick,black!60,arrows={-latex}] (0,0) -- (5.5,-2.75) node[right] {$(2,-1)$};
\fill[thick,black!60] (1,0) circle (3.0pt);
\fill[thick,black!60] (0,1) circle (3.0pt);
\fill[thick,black!60] (-1,0) circle (3.0pt);

\fill[black] (1,2) circle (3.0pt);
\draw[black,arrows={-latex},thick] (0,0) -- (1.65,3.3) node[above] {2};
\fill[black] (1,1) circle (3.0pt);
\draw[black,arrows={-latex},very thick] (0,0) -- (3.3,3.3) node[above] {1};
\fill[black] (3,2) circle (3.0pt);
\draw[black,arrows={-latex}] (0,0) -- (5.0,3.333) node[above right] {3};
\fill[black] (2,1) circle (3.0pt);
\draw[black,arrows={-latex},thick] (0,0) -- (5.5,2.75) node[right] {2};

\node at (0.60,2.50) {$\alpha_5$};
\node at (1.80,2.50) {$\alpha_4$};
\node at (3.30,2.65) {$\alpha_3$};
\node at (4.50,2.60) {$\alpha_2$};
\node at (4.50,1.50) {$\alpha_1$};

\fill[black] (1,-1) circle (3.0pt);
\draw[black,arrows={-latex},very thick] (0,0) -- (3.3,-3.3) node[below] {1};

\node at (-1.50,1.50) {$\beta_1$};

\fill[black] (3,-1) circle (3.0pt);
\draw[black,arrows={-latex},very thick] (0,0) -- (5.5,-1.833) node[right] {1};
\fill[black] (0,-1) circle (3.0pt);
\draw[black,arrows={-latex},thick] (0,0) -- (0,-3.5) node[below] {2};
\fill[black] (-1,-1) circle (3.0pt);
\draw[black,arrows={-latex}] (0,0) -- (-2.5,-2.5) node[left] {3};

\node at (-1.50,-0.70) {$\gamma_1$};
\node at (-1.00,-2.50) {$\gamma_2$};
\node at (1.00,-2.50) {$\gamma_3$};
\node at (3.50,-2.50) {$\gamma_4$};
\node at (4.40,-1.80) {$\delta_1$};
\node at (4.50,-0.60) {$\delta_2$};

%

\end{tikzpicture}
}
      \end{center}
      \end{subfigure}~
      \begin{subfigure}[b]{0.47\textwidth}
      \begin{center}
{\tiny
\begin{tikzpicture}[dotnode/.style={circle,draw=black,fill=none,inner sep=0pt,minimum size=3pt},node distance=8pt,yscale=0.5,xscale=0.7]

\node[dotnode,minimum size=5pt,fill=black!50,draw=black!50] (root) at (0,0) {};
\node[dotnode,fill=black!50,draw=black!50] (T) at (1.0,2.0) {};
\node[dotnode,fill=black!50,draw=black!50] (B) at (1.0,-2.0) {};

\node[dotnode,fill=black!50,draw=black!50] (I) at (2.0,3.0) {};
\node[dotnode] (II) at (2.0,1.0) {};
\node[dotnode,fill=black!50,draw=black!50] (III) at (2.0,-1.0) {};
\node[dotnode,fill=black!50,draw=black!50] (IV) at (2.0,-3.0) {};

\node[above of=I  ] {I  };
\node[above of=II ] {II };
\node[above of=III] {III};
\node[above of=IV ] {IV };

\node[right of=II] {$\beta_1$};

\draw[thick,black!50] (root) -- (T);
\draw[thick,black!50] (root) -- (B);
\draw[thick,black!50] (T) -- (I);
\draw[thick,black!50] (T) -- (II);
\draw[thick,black!50] (B) -- (III);
\draw[thick,black!50] (B) -- (IV);

\node[dotnode,fill] (Ir) at (3,3.5) {};
\node[dotnode] (Irr) at (4,3.8) {};
\node[dotnode,fill] (Irl) at (4,3.2) {};
\node[dotnode] (Irlr) at (5,3.4) {};
\node[dotnode] (Irll) at (5,3.0) {};
\node[dotnode,fill] (Il) at (3,2.3) {};
\node[dotnode] (Ilr) at (4,2.6) {};
\node[dotnode] (Ill) at (4,2.0) {};

\draw[thick] (I) -- (Ir); 
\draw[thick] (I) -- (Il); 
\draw[thick] (Ir) -- (Irr); 
\draw[thick] (Ir) -- (Irl); 
\draw[thick] (Irl) -- (Irlr); 
\draw[thick] (Irl) -- (Irll); 
\draw[thick] (Il) -- (Ilr); 
\draw[thick] (Il) -- (Ill); 

\node[right of=Irr] {$\alpha_1$};
\node[right of=Irlr] {$\alpha_2$};
\node[right of=Irll] {$\alpha_3$};
\node[right of=Ilr] {$\alpha_4$};
\node[right of=Ill] {$\alpha_5$};

\node[dotnode,fill] (IIIl) at (3,-0.5) {};
\node[dotnode,fill] (IIIll) at (4,-0.2) {};
\node[dotnode] (IIIlll) at (5,-0.0) {};
\node[dotnode] (IIIllr) at (5,-0.4) {};
\node[dotnode] (IIIlr) at (4,-0.8) {};
\node[dotnode] (IIIr) at (3,-1.4) {};

\draw[thick] (III) -- (IIIl); 
\draw[thick] (III) -- (IIIr); 
\draw[thick] (IIIl) -- (IIIll); 
\draw[thick] (IIIl) -- (IIIlr); 
\draw[thick] (IIIll) -- (IIIlll); 
\draw[thick] (IIIll) -- (IIIllr); 

\node[right of=IIIlll] {$\gamma_1$};
\node[right of=IIIllr] {$\gamma_2$};
\node[right of=IIIlr] {$\gamma_3$};
\node[right of=IIIr] {$\gamma_4$};

\node[dotnode] (IVl) at (3,-2.3) {};
\node[dotnode] (IVr) at (3,-3.5) {};

\draw[thick] (IV) -- (IVl); 
\draw[thick] (IV) -- (IVr); 

\node[right of=IVl] {$\delta_1$};
\node[right of=IVr] {$\delta_2$};

\end{tikzpicture}
}
      \end{center}
      \end{subfigure}
      \caption{A subdivision of $\F_2$ and the corresponding binary
        tree\label{fig:combinedtree}}
      \end{center}
\end{figure}
So, we get a bound of $\binom{n}{ k}4^k/k$ for the number of polygons
that come from a given $\F_a$. Since by Lemma~\ref{lem:Hirzebruch} $a<
n$, multiplying that bound for $n$ gives a global bound.
\end{proof}

In order to work towards the proof of Theorem~\ref{thm:smallpolygons},
we first show an example that illustrates two points. On the one hand
it shows that the upper bound given is not that bad; more precisely,
it shows that the maximum area of a smooth $k$-gon with $n$ boundary
points lies in $2^{\Theta(k)}n^{\Theta(1)}$.
On the other hand, it shows where the golden ratio in the statement
comes from.

On the other end of the range, Imre B\'ar\'any and Norihide
Tokushige \cite[Remark~2]{BaranyTokushige} constructed smooth lattice
$n$-gons with area less than $n^3/54$.

\begin{example}[A smooth $k$-gon with area $\Omega(\phi^{2k/3})$]
\label{exm:fibonacci}
We start with the normal fan of a unimodular triangle, whose rays we
label as follows:
\[
a_0=c_1=(1,0),\qquad
b_0=a_1=(0,1),\qquad
c_0=b_1=(-1,-1).
\]
Starting with this fan, we refine the three cones in an iterative and
symmetric manner. More precisely, choose an integer $\ell\ge 2$ and
introduce:
\[
a_i=a_{i-1}+a_{i-2},\quad 
b_i=b_{i-1}+b_{i-2},\quad 
c_i=c_{i-1}+c_{i-2},\qquad \forall i=2,\dots, \ell.
\]
Since, by symmetry, the sum of these $k=3\ell$ vectors is zero, this
is the normal fan of a smooth polygon with all edges of lattice length
$1$. The (normalized) area of this polygon is at least the determinant
of any pair of rays in the fan, since the convex hull of the
corresponding edges contains a triangle with that area. Let us
compute, for example, $\det(a_\ell,b_\ell)$.
By construction we have 
\[
a_\ell =(F_{\ell-1},F_\ell), \quad 
b_\ell =(-F_{\ell },-F_{\ell -2}),
\]
where $F_i$ denotes the $i$-th Fibonacci number. That is, $F_0=0$,
$F_1=F_2=1$, $F_3=2$, $F_{i+1} = F_{i-1} + F_i$. Hence, the
determinant we are interested in equals
\[
F_{\ell }^2-F_{\ell -1}F_{\ell -2}\simeq c \phi^{2\ell } =c \phi^{2k/3}
\]
for a certain constant $c$.
Since the perimeter of the polygon is $O(F_\ell )$, this lower bound
gives the correct area, up to the value of $c$.
\end{example}

\begin{proof}[Proof of Theorem~\ref{thm:smallpolygons}]
If the normal fan of $P$ is
the fan of $\P^2$, $P$ is a unimodular triangle dilated by
a factor of $n/3$, so its area is $n^2/9$.
Thus, assume that the normal fan $\Sigma$ of $P$ is a refinement of
the Hirzebruch fan $\F_a$ and, as in Lemma~\ref{lem:Hirzebruch},
assume that $a\ge 0$ is minimal with that property.
If $a\ne 0$, then the three rays $v_2=(0,1)$, $v_3=(-1,0)$ and
$v_4=(a,-1)$ are consecutive in the normal fan of $P$. Let $e_2$,
$e_3$ and $e_4$ be the corresponding edges. Then $P$ is inscribed in
the triangle with base $e_3$ and third vertex in the intersection of
the lines containing $e_2$ and $e_4$. That triangle (which may not be
a lattice triangle) has normalized area $l^2/(2a)\le n^2$, where $l\le
n$ is the length of $e_3$.

So, for the rest of the proof we assume $a=0$; that is, $\Sigma$
refines the fan $\F_0$ of $\P^1\times \P^1$. Let $k_1$, $k_2$, $k_3$
and $k_4$ denote the number of unimodular cones in $\Sigma$ that
refine the four cones of $\F_0$. 

The crucial observation is that, as in Example~\ref{exm:fibonacci},
in each quadrant, the $i$-th vector introduced by the refinement
process is bounded from above by the $i$-th Fibonacci number $F_i$
in each coordinate.
Here, as in Example~\ref{exm:fibonacci}, we reserve the
indices $i=0$ and $i=1$ for the two boundary primitive vectors in the
quadrant, so that the first vector refining the quadrant has $i=2$.
In particular, every coordinate of every ray is bounded above by
$F_{k-3}$, since $k=k_1+k_2+k_3+k_4$. On the other hand, the polygon
$P$ is contained in the zonotope obtained as the Minkowski sum of its
edges, and the (normalized) area of that zonotope is the sum of the
absolute values of the determinants of all pairs of rays in the fan,
where each ray is counted with a multiplicity equal to the length of
the corresponding edge~\cite[Ex.~7.19]{MR1311028}. The stated bound
then follows from these facts:
\begin{itemize}
\item The absolute value of each such determinant is bounded above by
  $2F_{k-3}^2$, which is smaller than $\phi^{2k}$.
\item The number of subdeterminants (counting rays with
  multiplicity) is bounded above by $\binom{n}{ 2}< n^2$.
   \qedhere
\end{itemize}
\end{proof}

\begin{remark}
We believe $n^2\phi^{2k/3}$ to be also an upper bound for the area,
which means that the construction of Example~\ref{exm:fibonacci} is
optimal, modulo a constant factor. The reason for this is that in
order to get the vectors in $\Sigma$ to sum up to zero (when counted
with multiplicity) we need to either have extremely high
multiplicities in some of them (making $n$ exponentially big) or have
at least two of the four cones of $\F_0$ be refined in basically the
same way (making the Fibonacci numbers involved bounded by $F_{k/2}$
rather than $F_k$). But if only two (opposite) cones of $\F_0$ have
this property then the Fibonacci-long vectors obtained will be almost
opposite, making the area small. Three of the cones need to have
vectors with big entries with respect to the basis of the starting
$\F_0$, which should give the bound of $\phi^{2k/3}$.
\end{remark}

\section{Classification in Dimension $3$}
\label{sec:classification}

This section summarizes the strategy to classify smooth $3$-polytopes
with at most $12$ lattice points. 
We don't follow the proof of Corollary~\ref{coro:smoothembeddings}
directly but use a modified strategy.
For full details, including source code,
see~\cite{LorenzDiplom,LorenzICMS}. In subsequent work, Anders
Lundman has extended this classification to $16$ lattice
points~\cite{LundmanMSc}.

\subsection{Generating Normal Fans}
Katsuya Miyake and Tadao Oda classified smooth $3$-dimensional
fans which are minimal with respect to equivariant
blow-ups~\cite[Theorem~1.34]{Oda}. This classification goes up to at
most eight rays or equivalently, $12$ full-dimensional cones.
Starting from this list, all possible sequences of blow-ups had to be
enumerated until no fan of a polytope with $\le 12$ lattice points
could occur further down the search tree.
In order to prune the search tree, we used bounds based on the
two-dimensional classification.

\subsection{Generating Polytopes}
The next step is to find the polytopes corresponding to ample
divisors, given the normal fan $\Sigma$. Let $\Sigma(1)$  denote the
set of rays in $\Sigma$, and for $b\in \R^{|\Sigma(1)|}$, we let
$$P_b = \{u \mid  \langle u, v_{\rho} \rangle \leq b_{\rho}\},$$  
where $v_{\rho}$ is the primitive generator of the ray $\rho \in
\Sigma(1)$. Note that for $b\in \Z^{|\Sigma(1)|}$, $P_b$ is the
lattice polytope corresponding to the torus invariant prime divisor
$D_b:= \sum_{\rho \in \Sigma(1)} b_{\rho} D_{\rho}$. For a curve $C$
on $X$, the function $\mathrm{Div}(X)_{\R} \to \R \colon  D \mapsto D
\cdot C$ is linear.
On a toric variety $D$ is ample if and only if $D \cdot C >0$ for all
torus invariant curves $C$, so these inequalities cut out the preimage
of the ample cone in $\mathrm{Div}(X)_{\R}$. This preimage consists of
the vectors $b$ such that the normal fan to $P_b$ is $\Sigma$. 
Note that when $D_b$ is ample, then $D_b\cdot C$ is the lattice length
of the edge of $P_b$ corresponding to the $(n-1)$-dimensional cone in
$\Sigma$ corresponding to  $C$, see \cite[1.4]{Laterveer96}.

Bounding the sum of the edge lengths $d(\mathcal{O}(D))$ as a lower
bound for the total number of lattice points, the search space of
possible $b$-vectors which yield at most $\nbr$ lattice points becomes
itself the set of lattice points in a polytope.

The last step is to remove all polytopes that are integrally equivalent
to another one in the list. 

All these computations can be done with the \polymake lattice polytope
package by Benjamin Lorenz, Andreas Paffenholz and Michael
Joswig~\cite{polymake,polymake-paper,polymakeLattice} using interfaces to \fourtitwo
by the 4ti2 team~\cite{4ti2}, \lattE by Jes\'us De Loera
et al.~\cite{latte,LattEsoftware} and \normaliz by Winfried Bruns et
al.~\cite{normaliz,normaliz2}.

\subsection{Classification Results}
\begin{theorem}
There are 41 equivalence classes of smooth lattice polygons with at
most 12 lattice points.
\begin{center}
\begin{tabular}{r|rrrrrrrrrrrrr}
Vertices	& $3$ &  $4$ & $5$ & $6$ & $7$ & $8$ & $9$ & $10$ & $11$ & $12$ \\
\hline
Polygons	& $3$ & $30$ & $3$ & $4$ & $0$ & $1$ & $0$ & $0$ & $0$ & $0$ \\
\end{tabular}
\end{center}
\end{theorem}

\begin{theorem}
There are 33 equivalence classes of smooth 3-di\-men\-sio\-nal
lattice polytopes with at most 12 lattice points.
\begin{center}
\begin{tabular}{r|rrrrr}
Vertices 		& $4$ & $6$ & $8$ & $10$ & $12$\\
\hline
Polytopes 			& $2$ & $25$ & $6$ & $0$ & $0$ \\
\end{tabular}
\end{center}
\end{theorem}
Note that a short parity argument shows that every \emph{simple} (and
hence every smooth) 3-polytope has an even number of vertices. 
Lists of all smooth polygons and smooth 3-polytopes with at most 12
lattice points can be found in the appendix.

\subsection{Comments} 
We now have a list of smooth lattice polytopes in dimensions two and 
three with at most $12$ lattice points. The bound $12$ may seem rather
low -- the smallest smooth 3-polytope with one interior lattice point
has $21$ lattice points total~\cite{KasprzykCanonical}. 
The classification carried out here serves as a proof of concept -- it
can be done. There are several points in the algorithm where it could
be improved (compare~\cite{LundmanMSc}).

In the current implementation, the generation of the normal fans is
the bottleneck. By implementing a different way to directly generate
all smooth normal fans one could skip the big recursion of calculating
all blowing-ups, as well as overcome the limits of at most $12$
vertices imposed by the Miyake/Oda classification.
The second point to work on is the calculation of lattice points of
the polytope containing all right-hand sides $b$. The dimension
of this polytope is equal to the Picard number of the toric variety:
the number of rays of the fan minus the ambient dimension.
Of course, better theoretical bounds for all steps of the algorithm
will directly improve the performance.

\subsection{Conjectures on smooth toric varieties}\label{sec:conjectures}

There is an entire hierarchy of successively stronger conjectures
concerning embeddings of smooth projective toric varieties which are
open even in dimension $3$,  (compare~\cite[p.~2313]{MFO07}). The
weakest conjecture is Oda's question whether every smooth lattice
polytope is \emph{integrally closed}, i.e., every lattice point in
$mP$ can be written as a sum of  $m$ lattice points in $P$.
The principal obstacle to theoretical progress on Oda's question on 
normality
and the related conjectures is a serious lack of well understood
examples. Recently, Gubeladze~\cite{GubeladzeLongEdges} has shown that
any lattice polytope with sufficiently long edges (depending on the
dimension) gives rise to a projectively normal embedding.
In view of this result, if there exists a counterexample, it is more
likely to be a small polytope. Yet, all polytopes in our
classification up to $12$ lattice points satisfy even the strongest of
these conjectures (see Corollary~\ref{cor:GB}).
In particular, the homogeneous coordinate ring is a Koszul algebra.

The following proposition shows that Oda's question implies
Theorem~\ref{thm:Qfactorialembeddings} for smooth toric varieties.

\begin{proposition}
There are only finitely many integrally closed lattice  polytopes $P$
with $\nbr$ lattice points.
\end{proposition}
\begin{proof}
  If $P$ is normal, then the semigroup in $M\times \Z$ generated by
  $(u,1)$, where $u$ is a lattice point in $P$, is normal. 
  This implies that the associated semigroup algebra is integrally
  closed and thus  a Cohen-Macaulay standard graded
  algebra~\cite{HochsterCM} with $\le \nbr$ generators. Thus, the
  coefficients of its Hilbert function (the Ehrhart polynomial of $P$)
  are bounded (compare, e.g.~\cite[Lemma~18.1]{HibiRedBook}).
  This bounds the degree (the normalized volume of $P$).
  By Theorem~\ref{thm:LZ}, there are only finitely many such $P$.
\end{proof}

Furthermore, using our classification, we were able to confirm the
strongest conjecture for smooth polytopes with at most 12 lattice
points.
\begin{theorem} \label{thm:GB}
  If $P$ is a 3-dimensional smooth polytope with at most $12$ lattice
  points, then $P$ has a regular unimodular triangulation with minimal
  non-faces of size two.
\end{theorem}

\begin{corollary} \label{cor:GB}
  Let $X$ be a smooth toric threefold embedded in $\P^{\le 11}$ using
  a complete linear series. Then the defining ideal of $X$ has an
  initial ideal generated by square-free quadratic monomials.
\end{corollary}

\bibliographystyle{alpha}
\bibliography{3dFiniteness}
\label{sec:biblio}

\clearpage

\appendix
{\tiny
\section*{List of Smooth Polygons with $\leq 12$ Lattice Points}

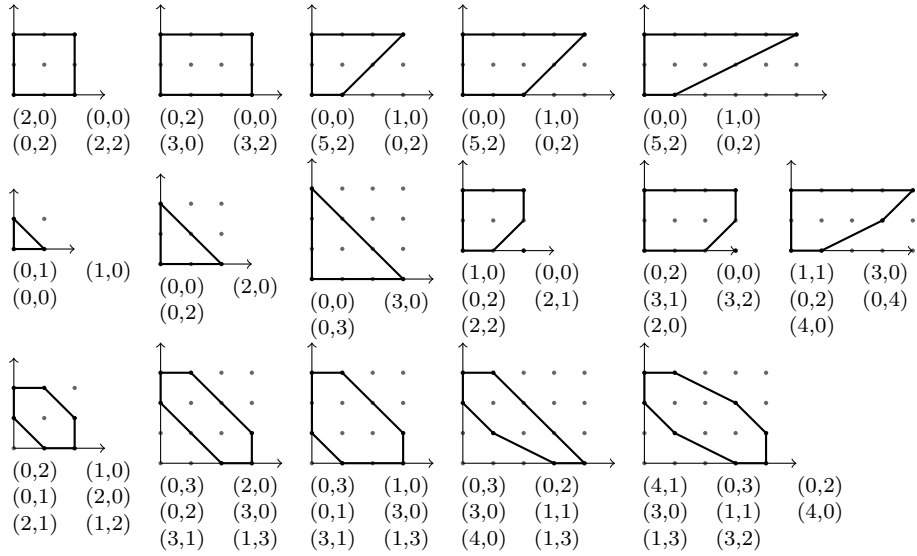
\begin{tabular}{llllll}

\begin{tabular}{@{}ll@{}}
\multicolumn{2}{@{}l@{}}{
\begin{tikzpicture}[scale=0.4]
\foreach \x in {0,...,2}
  \foreach \y in {0,...,2}
    \fill[black!60] (\x,\y) circle (2pt);
\draw[black,->] (0,0) -- (3,0);
\draw[black,->] (0,0) -- (0,3);
\fill (2,0) circle (2pt);
\fill (0,0) circle (2pt);
\fill (0,2) circle (2pt);
\fill (2,2) circle (2pt);
\draw[black,thick] (0,0) -- (0,2);
\draw[black,thick] (2,0) -- (0,0);
\draw[black,thick] (0,2) -- (2,2);
\draw[black,thick] (2,0) -- (2,2);
\end{tikzpicture}
}
\\
(2,0)&(0,0)\\ 
(0,2)&(2,2)\\ 
\end{tabular}
& 
\begin{tabular}{@{}ll@{}}
\multicolumn{2}{@{}l@{}}{
\begin{tikzpicture}[scale=0.4]
\foreach \x in {0,...,3}
  \foreach \y in {0,...,2}
    \fill[black!60] (\x,\y) circle (2pt);
\draw[black,->] (0,0) -- (4,0);
\draw[black,->] (0,0) -- (0,3);
\fill (0,2) circle (2pt);
\fill (0,0) circle (2pt);
\fill (3,0) circle (2pt);
\fill (3,2) circle (2pt);
\draw[black,thick] (0,2) -- (0,0);
\draw[black,thick] (0,0) -- (3,0);
\draw[black,thick] (0,2) -- (3,2);
\draw[black,thick] (3,0) -- (3,2);
\end{tikzpicture}
}
\\
(0,2)&(0,0)\\ 
(3,0)&(3,2)\\ 
\end{tabular}
& 
\begin{tabular}{@{}ll@{}}
\multicolumn{2}{@{}l@{}}{
\begin{tikzpicture}[scale=0.4]
\foreach \x in {0,...,3}
  \foreach \y in {0,...,2}
    \fill[black!60] (\x,\y) circle (2pt);
\draw[black,->] (0,0) -- (4,0);
\draw[black,->] (0,0) -- (0,3);
\fill (0,0) circle (2pt);
\fill (1,0) circle (2pt);
\fill (3,2) circle (2pt);
\fill (0,2) circle (2pt);
\draw[black,thick] (0,0) -- (0,2);
\draw[black,thick] (0,0) -- (1,0);
\draw[black,thick] (3,2) -- (0,2);
\draw[black,thick] (1,0) -- (3,2);
\end{tikzpicture}
}
\\
(0,0)&(1,0)\\ 
(5,2)&(0,2)\\ 
\end{tabular}
& 
\begin{tabular}{@{}ll@{}}
\multicolumn{2}{@{}l@{}}{
\begin{tikzpicture}[scale=0.4]
\foreach \x in {0,...,4}
  \foreach \y in {0,...,2}
    \fill[black!60] (\x,\y) circle (2pt);
\draw[black,->] (0,0) -- (5,0);
\draw[black,->] (0,0) -- (0,3);
\fill (0,0) circle (2pt);
\fill (2,0) circle (2pt);
\fill (4,2) circle (2pt);
\fill (0,2) circle (2pt);
\draw[black,thick] (0,0) -- (0,2);
\draw[black,thick] (0,0) -- (2,0);
\draw[black,thick] (4,2) -- (0,2);
\draw[black,thick] (2,0) -- (4,2);
\end{tikzpicture}
}
\\
(0,0)&(1,0)\\ 
(5,2)&(0,2)\\ 
\end{tabular}
& 
\multicolumn{2}{l}{
\begin{tabular}{@{}ll@{}}
\multicolumn{2}{@{}l@{}}{
\begin{tikzpicture}[scale=0.4]
\foreach \x in {0,...,5}
  \foreach \y in {0,...,2}
    \fill[black!60] (\x,\y) circle (2pt);
\draw[black,->] (0,0) -- (6,0);
\draw[black,->] (0,0) -- (0,3);
\fill (0,0) circle (2pt);
\fill (1,0) circle (2pt);
\fill (5,2) circle (2pt);
\fill (0,2) circle (2pt);
\draw[black,thick] (0,0) -- (0,2);
\draw[black,thick] (0,0) -- (1,0);
\draw[black,thick] (5,2) -- (0,2);
\draw[black,thick] (1,0) -- (5,2);
\end{tikzpicture}
}
\\
(0,0)&(1,0)\\ 
(5,2)&(0,2)\\ 
\end{tabular}
}
\\

\begin{tabular}{@{}ll@{}}
\multicolumn{2}{@{}l@{}}{
\begin{tikzpicture}[scale=0.4]
\foreach \x in {0,...,1}
  \foreach \y in {0,...,1}
    \fill[black!60] (\x,\y) circle (2pt);
\draw[black,->] (0,0) -- (2,0);
\draw[black,->] (0,0) -- (0,2);
\fill (0,1) circle (2pt);
\fill (1,0) circle (2pt);
\fill (0,0) circle (2pt);
\draw[black,thick] (0,1) -- (0,0);
\draw[black,thick] (1,0) -- (0,0);
\draw[black,thick] (0,1) -- (1,0);
\end{tikzpicture}
}
\\
(0,1)&(1,0)\\ 
(0,0)&\\ 
\end{tabular}
& 
\begin{tabular}{@{}ll@{}}
\multicolumn{2}{@{}l@{}}{
\begin{tikzpicture}[scale=0.4]
\foreach \x in {0,...,2}
  \foreach \y in {0,...,2}
    \fill[black!60] (\x,\y) circle (2pt);
\draw[black,->] (0,0) -- (3,0);
\draw[black,->] (0,0) -- (0,3);
\fill (0,0) circle (2pt);
\fill (2,0) circle (2pt);
\fill (0,2) circle (2pt);
\draw[black,thick] (0,0) -- (0,2);
\draw[black,thick] (0,0) -- (2,0);
\draw[black,thick] (2,0) -- (0,2);
\end{tikzpicture}
}
\\
(0,0)&(2,0)\\ 
(0,2)&\\ 
\end{tabular}
&
\begin{tabular}{@{}ll@{}}
\multicolumn{2}{@{}l@{}}{
\begin{tikzpicture}[scale=0.4]
\foreach \x in {0,...,3}
  \foreach \y in {0,...,3}
    \fill[black!60] (\x,\y) circle (2pt);
\draw[black,->] (0,0) -- (4,0);
\draw[black,->] (0,0) -- (0,4);
\fill (0,0) circle (2pt);
\fill (3,0) circle (2pt);
\fill (0,3) circle (2pt);
\draw[black,thick] (0,0) -- (0,3);
\draw[black,thick] (0,0) -- (3,0);
\draw[black,thick] (3,0) -- (0,3);
\end{tikzpicture}
}
\\
(0,0)&(3,0)\\ 
(0,3)&\\ 
\end{tabular}
& 


\begin{tabular}{@{}ll@{}}
\multicolumn{2}{@{}l@{}}{
\begin{tikzpicture}[scale=0.4]
\foreach \x in {0,...,2}
  \foreach \y in {0,...,2}
    \fill[black!60] (\x,\y) circle (2pt);
\draw[black,->] (0,0) -- (3,0);
\draw[black,->] (0,0) -- (0,3);
\fill (2,0) circle (2pt);
\fill (0,0) circle (2pt);
\fill (0,2) circle (2pt);
\fill (2,2) circle (2pt);
\draw[black,thick] (0,0) -- (0,2);
\draw[black,thick] (1,0) -- (0,0);
\draw[black,thick] (0,2) -- (2,2);
\draw[black,thick] (2,1) -- (2,2);
\draw[black,thick] (2,1) -- (1,0);
\end{tikzpicture}
}
\\
(1,0)&(0,0)\\ 
(0,2)&(2,1)\\
(2,2)\\ 
\end{tabular}
& 
\begin{tabular}{@{}ll@{}}
\multicolumn{2}{@{}l@{}}{
\begin{tikzpicture}[scale=0.4]
\foreach \x in {0,...,3}
  \foreach \y in {0,...,2}
    \fill[black!60] (\x,\y) circle (2pt);
\draw[black,->] (0,0) -- (3,0);
\draw[black,->] (0,0) -- (0,3);
\fill (0,2) circle (2pt);
\fill (0,0) circle (2pt);
\fill (3,0) circle (2pt);
\fill (3,2) circle (2pt);
\draw[black,thick] (0,2) -- (0,0);
\draw[black,thick] (0,0) -- (2,0);
\draw[black,thick] (0,2) -- (3,2);
\draw[black,thick] (3,1) -- (3,2);
\draw[black,thick] (3,1) -- (2,0);
\end{tikzpicture}
}
\\
(0,2)&(0,0)\\ 
(3,1)&(3,2)\\
(2,0)\\ 
\end{tabular}
&
\begin{tabular}{@{}ll@{}}
\multicolumn{2}{@{}l@{}}{
\begin{tikzpicture}[scale=0.4]
\foreach \x in {0,...,4}
  \foreach \y in {0,...,2}
    \fill[black!60] (\x,\y) circle (2pt);
\draw[black,->] (0,0) -- (4,0);
\draw[black,->] (0,0) -- (0,3);
\fill (1,0) circle (2pt);
\fill (3,1) circle (2pt);
\fill (0,0) circle (2pt);
\fill (0,2) circle (2pt);
\fill (4,2) circle (2pt);
\draw[black,thick] (0,0) -- (0,2);
\draw[black,thick] (3,1) -- (4,2);
\draw[black,thick] (0,2) -- (4,2);
\draw[black,thick] (1,0) -- (0,0);
\draw[black,thick] (1,0) -- (3,1);
\end{tikzpicture}
}
\\
(1,1)&(3,0)\\ 
(0,2)&(0,4)\\ 
(4,0)&\\ 
\end{tabular}
\\

\begin{tabular}{@{}ll@{}}
\multicolumn{2}{@{}l@{}}{
\begin{tikzpicture}[scale=0.4]
\foreach \x in {0,...,2}
  \foreach \y in {0,...,2}
    \fill[black!60] (\x,\y) circle (2pt);
\draw[black,->] (0,0) -- (3,0);
\draw[black,->] (0,0) -- (0,3);
\fill (0,2) circle (2pt);
\fill (1,0) circle (2pt);
\fill (0,1) circle (2pt);
\fill (2,0) circle (2pt);
\fill (2,1) circle (2pt);
\fill (1,2) circle (2pt);
\draw[black,thick] (0,2) -- (0,1);
\draw[black,thick] (1,0) -- (2,0);
\draw[black,thick] (2,1) -- (1,2);
\draw[black,thick] (1,0) -- (0,1);
\draw[black,thick] (2,0) -- (2,1);
\draw[black,thick] (0,2) -- (1,2);
\end{tikzpicture}
}
\\
(0,2)&(1,0)\\ 
(0,1)&(2,0)\\ 
(2,1)&(1,2)\\ 
\end{tabular}
& 
\begin{tabular}{@{}ll@{}}
\multicolumn{2}{@{}l@{}}{
\begin{tikzpicture}[scale=0.4]
\foreach \x in {0,...,3}
  \foreach \y in {0,...,3}
    \fill[black!60] (\x,\y) circle (2pt);
\draw[black,->] (0,0) -- (4,0);
\draw[black,->] (0,0) -- (0,4);
\fill (0,3) circle (2pt);
\fill (2,0) circle (2pt);
\fill (0,2) circle (2pt);
\fill (3,0) circle (2pt);
\fill (3,1) circle (2pt);
\fill (1,3) circle (2pt);
\draw[black,thick] (0,3) -- (0,2);
\draw[black,thick] (2,0) -- (3,0);
\draw[black,thick] (3,1) -- (1,3);
\draw[black,thick] (2,0) -- (0,2);
\draw[black,thick] (3,0) -- (3,1);
\draw[black,thick] (0,3) -- (1,3);
\end{tikzpicture}
}
\\
(0,3)&(2,0)\\ 
(0,2)&(3,0)\\ 
(3,1)&(1,3)\\ 
\end{tabular}
& 
\begin{tabular}{@{}ll@{}}
\multicolumn{2}{@{}l@{}}{
\begin{tikzpicture}[scale=0.4]
\foreach \x in {0,...,3}
  \foreach \y in {0,...,3}
    \fill[black!60] (\x,\y) circle (2pt);
\draw[black,->] (0,0) -- (4,0);
\draw[black,->] (0,0) -- (0,4);
\fill (0,3) circle (2pt);
\fill (1,0) circle (2pt);
\fill (0,1) circle (2pt);
\fill (3,0) circle (2pt);
\fill (3,1) circle (2pt);
\fill (1,3) circle (2pt);
\draw[black,thick] (0,3) -- (0,1);
\draw[black,thick] (1,0) -- (3,0);
\draw[black,thick] (3,1) -- (1,3);
\draw[black,thick] (1,0) -- (0,1);
\draw[black,thick] (3,0) -- (3,1);
\draw[black,thick] (0,3) -- (1,3);
\end{tikzpicture}
}
\\
(0,3)&(1,0)\\ 
(0,1)&(3,0)\\ 
(3,1)&(1,3)\\ 
\end{tabular}
&
\begin{tabular}{@{}ll@{}}
\multicolumn{2}{@{}l@{}}{
\begin{tikzpicture}[scale=0.4]
\foreach \x in {0,...,4}
  \foreach \y in {0,...,3}
    \fill[black!60] (\x,\y) circle (2pt);
\draw[black,->] (0,0) -- (5,0);
\draw[black,->] (0,0) -- (0,4);
\fill (0,3) circle (2pt);
\fill (0,2) circle (2pt);
\fill (3,0) circle (2pt);
\fill (1,1) circle (2pt);
\fill (4,0) circle (2pt);
\fill (1,3) circle (2pt);
\draw[black,thick] (0,3) -- (0,2);
\draw[black,thick] (3,0) -- (4,0);
\draw[black,thick] (4,0) -- (1,3);
\draw[black,thick] (0,2) -- (1,1);
\draw[black,thick] (3,0) -- (1,1);
\draw[black,thick] (0,3) -- (1,3);
\end{tikzpicture}
}
\\
(0,3)&(0,2)\\ 
(3,0)&(1,1)\\ 
(4,0)&(1,3)\\ 
\end{tabular}
& 

\multicolumn{2}{l}{
\begin{tabular}{@{}ll@{}ll@{}}
\multicolumn{2}{@{}l@{}}{
\begin{tikzpicture}[scale=0.4]
\foreach \x in {0,...,4}
  \foreach \y in {0,...,3}
    \fill[black!60] (\x,\y) circle (2pt);
\draw[black,->] (0,0) -- (5,0);
\draw[black,->] (0,0) -- (0,4);
\fill (4,1) circle (2pt);
\fill (0,3) circle (2pt);
\fill (0,2) circle (2pt);
\fill (3,0) circle (2pt);
\fill (1,1) circle (2pt);
\fill (4,0) circle (2pt);
\fill (1,3) circle (2pt);
\fill (3,2) circle (2pt);
\draw[black,thick] (0,3) -- (0,2);
\draw[black,thick] (3,0) -- (4,0);
\draw[black,thick] (4,1) -- (3,2);
\draw[black,thick] (0,2) -- (1,1);
\draw[black,thick] (3,0) -- (1,1);
\draw[black,thick] (4,1) -- (4,0);
\draw[black,thick] (0,3) -- (1,3);
\draw[black,thick] (1,3) -- (3,2);
\end{tikzpicture}
}
\\
(4,1)&(0,3)&(0,2)\\
(3,0)&(1,1)&(4,0)\\ 
(1,3)&(3,2)\\ 
\end{tabular}
}\\ 
\end{tabular}
\bigskip

There are also $25$ Hirzebruch quadrangles, omitted from the list: 
\[
Q_{a,b}:=\conv\{(0,0), (0,a), (1,0), (1,b)\},
\]
for all $a,b\ge 1$ with $a+b\le 10$:
\[
(a,b)\in \left\{
\begin{tabular}{c}
(1,1), (1,2), (1,3), (1,4), (1,5), (1,6), (1,7), (1,8), (1,9)\\
(2,2), (2,3), (2,4), (2,5), (2,6), (2,7), (2,8)\\
(3,3), (3,4), (3,5), (3,6), (3,7),
(4,4), (4,5), (4,6), 
(5,5)\\
\end{tabular}
\right\}.
\]

\vfill

\section*{List of Smooth 3-Polytopes with $\leq 12$ Lattice Points}

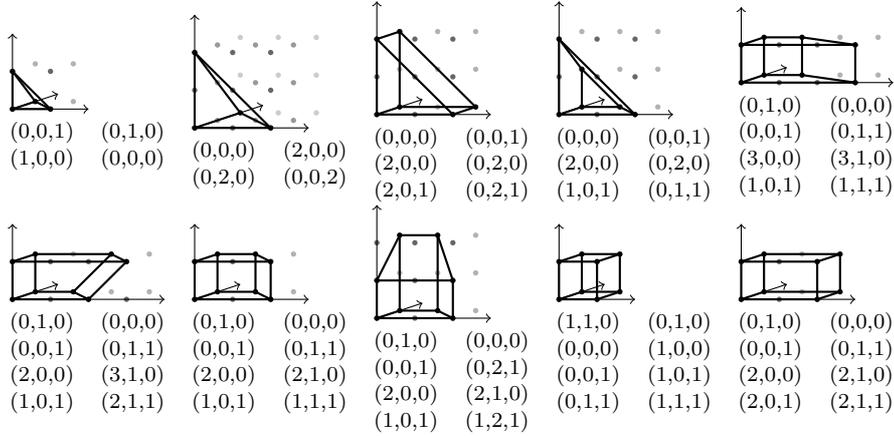
\begin{tabular}{llllll}

\begin{tabular}{@{}ll@{}}
\multicolumn{2}{@{}l@{}}{
\begin{tikzpicture}[scale=0.5]
\foreach \x in {0,...,1}
  \foreach \y in {0,...,1} {
\fill[black!30] (\x+0.6,\y+0.2) circle (2pt);
}
\foreach \x in {0,...,1}
  \foreach \y in {0,...,1}
    \fill[black!60] (\x,\y) circle (2pt);
\draw[black,->] (0,0) -- (2,0);
\draw[black,->] (0,0) -- (0,2);
\draw[black,->] (0,0) -- (6/5,2/5);
\fill (3/5,1/5) circle (2pt);
\fill (0,1) circle (2pt);
\fill (1,0) circle (2pt);
\fill (0,0) circle (2pt);
\draw[black,thick] (3/5,1/5) -- (0,1);
\draw[black,thick] (3/5,1/5) -- (1,0);
\draw[black,thick] (3/5,1/5) -- (0,0);
\draw[black,thick] (0,1) -- (1,0);
\draw[black,thick] (0,1) -- (0,0);
\draw[black,thick] (1,0) -- (0,0);
\end{tikzpicture}
}
\\
(0,0,1)&(0,1,0)\\ 
(1,0,0)&(0,0,0)\\ 
\end{tabular}
& 
\begin{tabular}{@{}ll@{}}
\multicolumn{2}{@{}l@{}}{
\begin{tikzpicture}[scale=0.5]
\foreach \x in {0,...,2}
  \foreach \y in {0,...,2} {
\fill[black!20] (\x+1.2,\y+0.4) circle (2pt);
\fill[black!40] (\x+0.6,\y+0.2) circle (2pt);
}
\foreach \x in {0,...,2}
  \foreach \y in {0,...,2}
    \fill[black!60] (\x,\y) circle (2pt);
\draw[black,->] (0,0) -- (3,0);
\draw[black,->] (0,0) -- (0,3);
\draw[black,->] (0,0) -- (9/5,3/5);
\fill (0,0) circle (2pt);
\fill (2,0) circle (2pt);
\fill (0,2) circle (2pt);
\fill (6/5,2/5) circle (2pt);
\draw[black,thick] (0,0) -- (2,0);
\draw[black,thick] (0,0) -- (0,2);
\draw[black,thick] (0,0) -- (6/5,2/5);
\draw[black,thick] (2,0) -- (0,2);
\draw[black,thick] (2,0) -- (6/5,2/5);
\draw[black,thick] (0,2) -- (6/5,2/5);
\end{tikzpicture}
}
\\
(0,0,0)&(2,0,0)\\ 
(0,2,0)&(0,0,2)\\ 
\end{tabular}
&

\begin{tabular}{@{}ll@{}}
\multicolumn{2}{@{}l@{}}{
\begin{tikzpicture}[scale=0.5]
\foreach \x in {0,...,2}
  \foreach \y in {0,...,2} {
\fill[black!30] (\x+0.6,\y+0.2) circle (2pt);
}
\foreach \x in {0,...,2}
  \foreach \y in {0,...,2}
    \fill[black!60] (\x,\y) circle (2pt);
\draw[black,->] (0,0) -- (3,0);
\draw[black,->] (0,0) -- (0,3);
\draw[black,->] (0,0) -- (6/5,2/5);
\fill (0,0) circle (2pt);
\fill (3/5,1/5) circle (2pt);
\fill (2,0) circle (2pt);
\fill (0,2) circle (2pt);
\fill (13/5,1/5) circle (2pt);
\fill (3/5,11/5) circle (2pt);
\draw[black,thick] (0,0) -- (3/5,1/5);
\draw[black,thick] (0,0) -- (0,2);
\draw[black,thick] (3/5,1/5) -- (3/5,11/5);
\draw[black,thick] (0,2) -- (3/5,11/5);
\draw[black,thick] (0,0) -- (2,0);
\draw[black,thick] (3/5,1/5) -- (13/5,1/5);
\draw[black,thick] (2,0) -- (13/5,1/5);
\draw[black,thick] (2,0) -- (0,2);
\draw[black,thick] (13/5,1/5) -- (3/5,11/5);
\end{tikzpicture}
}
\\
(0,0,0)&(0,0,1)\\ 
(2,0,0)&(0,2,0)\\ 
(2,0,1)&(0,2,1)\\ 
\end{tabular}
&

\begin{tabular}{@{}ll@{}}
\multicolumn{2}{@{}l@{}}{
\begin{tikzpicture}[scale=0.5]
\foreach \x in {0,...,2}
  \foreach \y in {0,...,2} {
\fill[black!30] (\x+0.6,\y+0.2) circle (2pt);
}
\foreach \x in {0,...,2}
  \foreach \y in {0,...,2}
    \fill[black!60] (\x,\y) circle (2pt);
\draw[black,->] (0,0) -- (3,0);
\draw[black,->] (0,0) -- (0,3);
\draw[black,->] (0,0) -- (6/5,2/5);
\fill (0,0) circle (2pt);
\fill (3/5,1/5) circle (2pt);
\fill (2,0) circle (2pt);
\fill (0,2) circle (2pt);
\fill (8/5,1/5) circle (2pt);
\fill (3/5,6/5) circle (2pt);
\draw[black,thick] (0,0) -- (3/5,1/5);
\draw[black,thick] (0,0) -- (0,2);
\draw[black,thick] (3/5,1/5) -- (3/5,6/5);
\draw[black,thick] (0,2) -- (3/5,6/5);
\draw[black,thick] (0,0) -- (2,0);
\draw[black,thick] (3/5,1/5) -- (8/5,1/5);
\draw[black,thick] (2,0) -- (8/5,1/5);
\draw[black,thick] (2,0) -- (0,2);
\draw[black,thick] (8/5,1/5) -- (3/5,6/5);
\end{tikzpicture}
}
\\
(0,0,0)&(0,0,1)\\ 
(2,0,0)&(0,2,0)\\ 
(1,0,1)&(0,1,1)\\ 
\end{tabular}
& 
\begin{tabular}{@{}ll@{}}
\multicolumn{2}{@{}l@{}}{
\begin{tikzpicture}[scale=0.5]
\foreach \x in {0,...,3}
  \foreach \y in {0,...,1} {
\fill[black!30] (\x+0.6,\y+0.2) circle (2pt);
}
\foreach \x in {0,...,3}
  \foreach \y in {0,...,1}
    \fill[black!60] (\x,\y) circle (2pt);
\draw[black,->] (0,0) -- (4,0);
\draw[black,->] (0,0) -- (0,2);
\draw[black,->] (0,0) -- (6/5,2/5);
\fill (0,1) circle (2pt);
\fill (0,0) circle (2pt);
\fill (3/5,1/5) circle (2pt);
\fill (3/5,6/5) circle (2pt);
\fill (3,0) circle (2pt);
\fill (3,1) circle (2pt);
\fill (8/5,1/5) circle (2pt);
\fill (8/5,6/5) circle (2pt);
\draw[black,thick] (0,0) -- (3/5,1/5);
\draw[black,thick] (0,1) -- (0,0);
\draw[black,thick] (3/5,1/5) -- (3/5,6/5);
\draw[black,thick] (0,1) -- (3/5,6/5);
\draw[black,thick] (0,0) -- (3,0);
\draw[black,thick] (3/5,1/5) -- (8/5,1/5);
\draw[black,thick] (3,0) -- (8/5,1/5);
\draw[black,thick] (0,1) -- (3,1);
\draw[black,thick] (3,0) -- (3,1);
\draw[black,thick] (3/5,6/5) -- (8/5,6/5);
\draw[black,thick] (8/5,1/5) -- (8/5,6/5);
\draw[black,thick] (3,1) -- (8/5,6/5);
\end{tikzpicture}
}
\\
(0,1,0)&(0,0,0)\\ 
(0,0,1)&(0,1,1)\\ 
(3,0,0)&(3,1,0)\\ 
(1,0,1)&(1,1,1)\\ 
\end{tabular}
\\ 
\begin{tabular}{@{}ll@{}}
\multicolumn{2}{@{}l@{}}{
\begin{tikzpicture}[scale=0.5]
\foreach \x in {0,...,3}
  \foreach \y in {0,...,1} {
\fill[black!30] (\x+0.6,\y+0.2) circle (2pt);
}
\foreach \x in {0,...,3}
  \foreach \y in {0,...,1}
    \fill[black!60] (\x,\y) circle (2pt);
\draw[black,->] (0,0) -- (4,0);
\draw[black,->] (0,0) -- (0,2);
\draw[black,->] (0,0) -- (6/5,2/5);
\fill (0,1) circle (2pt);
\fill (0,0) circle (2pt);
\fill (3/5,1/5) circle (2pt);
\fill (3/5,6/5) circle (2pt);
\fill (2,0) circle (2pt);
\fill (3,1) circle (2pt);
\fill (8/5,1/5) circle (2pt);
\fill (13/5,6/5) circle (2pt);
\draw[black,thick] (0,0) -- (3/5,1/5);
\draw[black,thick] (0,1) -- (0,0);
\draw[black,thick] (3/5,1/5) -- (3/5,6/5);
\draw[black,thick] (0,1) -- (3/5,6/5);
\draw[black,thick] (0,0) -- (2,0);
\draw[black,thick] (3/5,1/5) -- (8/5,1/5);
\draw[black,thick] (2,0) -- (8/5,1/5);
\draw[black,thick] (0,1) -- (3,1);
\draw[black,thick] (2,0) -- (3,1);
\draw[black,thick] (3/5,6/5) -- (13/5,6/5);
\draw[black,thick] (8/5,1/5) -- (13/5,6/5);
\draw[black,thick] (3,1) -- (13/5,6/5);
\end{tikzpicture}
}
\\
(0,1,0)&(0,0,0)\\ 
(0,0,1)&(0,1,1)\\ 
(2,0,0)&(3,1,0)\\ 
(1,0,1)&(2,1,1)\\ 
\end{tabular}
& 
\begin{tabular}{@{}ll@{}}
\multicolumn{2}{@{}l@{}}{
\begin{tikzpicture}[scale=0.5]
\foreach \x in {0,...,2}
  \foreach \y in {0,...,1} {
\fill[black!30] (\x+0.6,\y+0.2) circle (2pt);
}
\foreach \x in {0,...,2}
  \foreach \y in {0,...,1}
    \fill[black!60] (\x,\y) circle (2pt);
\draw[black,->] (0,0) -- (3,0);
\draw[black,->] (0,0) -- (0,2);
\draw[black,->] (0,0) -- (6/5,2/5);
\fill (0,1) circle (2pt);
\fill (0,0) circle (2pt);
\fill (3/5,1/5) circle (2pt);
\fill (3/5,6/5) circle (2pt);
\fill (2,0) circle (2pt);
\fill (2,1) circle (2pt);
\fill (8/5,1/5) circle (2pt);
\fill (8/5,6/5) circle (2pt);
\draw[black,thick] (0,0) -- (3/5,1/5);
\draw[black,thick] (0,1) -- (0,0);
\draw[black,thick] (3/5,1/5) -- (3/5,6/5);
\draw[black,thick] (0,1) -- (3/5,6/5);
\draw[black,thick] (0,0) -- (2,0);
\draw[black,thick] (3/5,1/5) -- (8/5,1/5);
\draw[black,thick] (2,0) -- (8/5,1/5);
\draw[black,thick] (0,1) -- (2,1);
\draw[black,thick] (2,0) -- (2,1);
\draw[black,thick] (3/5,6/5) -- (8/5,6/5);
\draw[black,thick] (8/5,1/5) -- (8/5,6/5);
\draw[black,thick] (2,1) -- (8/5,6/5);
\end{tikzpicture}
}
\\
(0,1,0)&(0,0,0)\\ 
(0,0,1)&(0,1,1)\\ 
(2,0,0)&(2,1,0)\\ 
(1,0,1)&(1,1,1)\\ 
\end{tabular}
& 
\begin{tabular}{@{}ll@{}}
\multicolumn{2}{@{}l@{}}{
\begin{tikzpicture}[scale=0.5]
\foreach \x in {0,...,2}
  \foreach \y in {0,...,2} {
\fill[black!30] (\x+0.6,\y+0.2) circle (2pt);
}
\foreach \x in {0,...,2}
  \foreach \y in {0,...,2}
    \fill[black!60] (\x,\y) circle (2pt);
\draw[black,->] (0,0) -- (3,0);
\draw[black,->] (0,0) -- (0,3);
\draw[black,->] (0,0) -- (6/5,2/5);
\fill (0,1) circle (2pt);
\fill (0,0) circle (2pt);
\fill (3/5,1/5) circle (2pt);
\fill (3/5,11/5) circle (2pt);
\fill (2,0) circle (2pt);
\fill (2,1) circle (2pt);
\fill (8/5,1/5) circle (2pt);
\fill (8/5,11/5) circle (2pt);
\draw[black,thick] (0,0) -- (3/5,1/5);
\draw[black,thick] (0,1) -- (0,0);
\draw[black,thick] (3/5,1/5) -- (3/5,11/5);
\draw[black,thick] (0,1) -- (3/5,11/5);
\draw[black,thick] (0,0) -- (2,0);
\draw[black,thick] (3/5,1/5) -- (8/5,1/5);
\draw[black,thick] (2,0) -- (8/5,1/5);
\draw[black,thick] (0,1) -- (2,1);
\draw[black,thick] (2,0) -- (2,1);
\draw[black,thick] (3/5,11/5) -- (8/5,11/5);
\draw[black,thick] (8/5,1/5) -- (8/5,11/5);
\draw[black,thick] (2,1) -- (8/5,11/5);
\end{tikzpicture}
}
\\
(0,1,0)&(0,0,0)\\ 
(0,0,1)&(0,2,1)\\ 
(2,0,0)&(2,1,0)\\ 
(1,0,1)&(1,2,1)\\ 
\end{tabular}
& 
\begin{tabular}{@{}ll@{}}
\multicolumn{2}{@{}l@{}}{
\begin{tikzpicture}[scale=0.5]
\foreach \x in {0,...,1}
  \foreach \y in {0,...,1} {
\fill[black!30] (\x+0.6,\y+0.2) circle (2pt);
}
\foreach \x in {0,...,1}
  \foreach \y in {0,...,1}
    \fill[black!60] (\x,\y) circle (2pt);
\draw[black,->] (0,0) -- (2,0);
\draw[black,->] (0,0) -- (0,2);
\draw[black,->] (0,0) -- (6/5,2/5);
\fill (1,1) circle (2pt);
\fill (0,1) circle (2pt);
\fill (0,0) circle (2pt);
\fill (1,0) circle (2pt);
\fill (3/5,1/5) circle (2pt);
\fill (8/5,1/5) circle (2pt);
\fill (3/5,6/5) circle (2pt);
\fill (8/5,6/5) circle (2pt);
\draw[black,thick] (0,0) -- (3/5,1/5);
\draw[black,thick] (0,1) -- (0,0);
\draw[black,thick] (3/5,1/5) -- (3/5,6/5);
\draw[black,thick] (0,1) -- (3/5,6/5);
\draw[black,thick] (0,0) -- (1,0);
\draw[black,thick] (3/5,1/5) -- (8/5,1/5);
\draw[black,thick] (1,0) -- (8/5,1/5);
\draw[black,thick] (1,1) -- (0,1);
\draw[black,thick] (1,1) -- (1,0);
\draw[black,thick] (3/5,6/5) -- (8/5,6/5);
\draw[black,thick] (8/5,1/5) -- (8/5,6/5);
\draw[black,thick] (1,1) -- (8/5,6/5);
\end{tikzpicture}
}
\\
(1,1,0)&(0,1,0)\\ 
(0,0,0)&(1,0,0)\\ 
(0,0,1)&(1,0,1)\\ 
(0,1,1)&(1,1,1)\\ 
\end{tabular}
& 
\begin{tabular}{@{}ll@{}}
\multicolumn{2}{@{}l@{}}{
\begin{tikzpicture}[scale=0.5]
\foreach \x in {0,...,2}
  \foreach \y in {0,...,1} {
\fill[black!30] (\x+0.6,\y+0.2) circle (2pt);
}
\foreach \x in {0,...,2}
  \foreach \y in {0,...,1}
    \fill[black!60] (\x,\y) circle (2pt);
\draw[black,->] (0,0) -- (3,0);
\draw[black,->] (0,0) -- (0,2);
\draw[black,->] (0,0) -- (6/5,2/5);
\fill (0,1) circle (2pt);
\fill (0,0) circle (2pt);
\fill (3/5,1/5) circle (2pt);
\fill (3/5,6/5) circle (2pt);
\fill (2,0) circle (2pt);
\fill (2,1) circle (2pt);
\fill (13/5,1/5) circle (2pt);
\fill (13/5,6/5) circle (2pt);
\draw[black,thick] (0,0) -- (3/5,1/5);
\draw[black,thick] (0,1) -- (0,0);
\draw[black,thick] (3/5,1/5) -- (3/5,6/5);
\draw[black,thick] (0,1) -- (3/5,6/5);
\draw[black,thick] (0,0) -- (2,0);
\draw[black,thick] (3/5,1/5) -- (13/5,1/5);
\draw[black,thick] (2,0) -- (13/5,1/5);
\draw[black,thick] (0,1) -- (2,1);
\draw[black,thick] (2,0) -- (2,1);
\draw[black,thick] (3/5,6/5) -- (13/5,6/5);
\draw[black,thick] (13/5,1/5) -- (13/5,6/5);
\draw[black,thick] (2,1) -- (13/5,6/5);
\end{tikzpicture}
}
\\
(0,1,0)&(0,0,0)\\ 
(0,0,1)&(0,1,1)\\ 
(2,0,0)&(2,1,0)\\ 
(2,0,1)&(2,1,1)\\ 
\end{tabular}
& 
\end{tabular}

\bigskip

There are also the following $23$ prisms, omitted from the list:
\[
Q_{a,b,c}:=\conv\{(0,0,0), (0,0,a), (1,0,0), (1,0,b), (0,1,0), (0,1,c)\},
\]
for all $a,b,c\ge 1$ with $a+b+c\le 9$:
\[
(a,b,c)\in \left\{
\begin{tabular}{c}
(1,1,1), (1,1,2), (1,1,3), (1,1,4), (1,1,5), (1,1,6), (1,1,7)\\
(1,2,2), (1,2,3), (1,2,4), (1,2,5), (1,2,6),
(1,3,3), (1,3,4), (1,3,5), (1,4,4)\\
(2,2,2), (2,2,3), (2,2,4), (2,2,5),
(2,3,3), (2,3,4), (3,3,3)\\
\end{tabular}
\right\}.
\]

\vfill
}

\end{document}